\documentclass[
a4paper,11pt]{amsart}
\usepackage[english]{babel}
\usepackage{url}
\usepackage[
colorlinks=true, linkcolor = black, urlcolor=black, citecolor=blue, anchorcolor=blue]{hyperref}
\usepackage{comment}
\usepackage{tcolorbox}

\usepackage{cancel}

\newcommand{\crosses}[1]{%
	\ifcase#1\relax
	\or
	\rslash\or
	\rslash\mskip-5.5mu\rslash\or
	\rslash\mskip-5.5mu\rslash\mskip-5.5mu\rslash%
	\fi
}
\newcommand{\rslash}{\raisebox{.15ex}{/}}
\usepackage{xcolor,color}
\usepackage{geometry}
\usepackage[all, cmtip]{xy}
\usepackage{enumitem}
\usepackage{booktabs}
\usepackage{slashed}
\usepackage{bbm}
\usepackage{cancel}

\usepackage{cleveref}
\usepackage{tikz}
\usepackage{tikz-cd}
\tikzcdset{arrow style=tikz, diagrams={>=stealth}}

\usepackage{xargs}  
\usepackage{soul}

\usepackage{amsmath}
\usepackage{amssymb,graphicx}
\usepackage{amsthm}
\usepackage{latexsym}
\usepackage{amsfonts}
\usepackage{bbm}
\usepackage{mathrsfs}
\usepackage{cases}

\usepackage{etex}

\calclayout

\numberwithin{equation}{section}

\theoremstyle{plain}
\newtheorem{lemma}{Lemma}[section]
\newtheorem{proposition}[lemma]{Proposition}
\newtheorem{proposition/definition}[lemma]{Proposition/Definition}
\newtheorem{theorem}[lemma]{Theorem}

\theoremstyle{definition}
\newtheorem{definition}[lemma]{Definition}
\newtheorem{remark}[lemma]{Remark}
\newtheorem{example}[lemma]{Example}

\DeclareRobustCommand{\SkipTocEntry}[5]{}

\renewcommand{\theta}{\vartheta}

\allowdisplaybreaks

\title{Fiber-wise Linear Differential Operators}

\makeatletter
\def\author@andify{%
	\nxandlist {\unskip ,\penalty-1 \space\ignorespaces}%
	{\unskip {} \@@and~}%
	{\unskip \penalty-2 \space \@@and~}%
}
\makeatother

\author[F.~Pugliese]{Fabrizio Pugliese}
\address{DipMat, Universit\`a degli Studi di Salerno, via Giovanni Paolo II n${}^{\circ}$123, 84084 Fisciano (SA), Italy.}
\email{\href{mailto:fpugliese@unisa.it}{fpugliese@unisa.it}}

\author[G.~Sparano]{Giovanni Sparano}
\address{DipMat, Universit\`a degli Studi di Salerno, via Giovanni Paolo II n${}^{\circ}$123, 84084 Fisciano (SA), Italy.}
\email{\href{mailto:sparano@unisa.it}{sparano@unisa.it}}

\author[L. Vitagliano]{Luca Vitagliano}
\address{DipMat, Università degli Studi di Salerno, via Giovanni Paolo II n◦ 123, 84084 Fisciano (SA), Italy.}
\email{\href{mailto:lvitagliano@unisa.it}{lvitagliano@unisa.it}}



\begin{document}

\begin{abstract}
We define a new notion of \emph{fiber-wise linear differential operator} on the total space of a vector bundle $E$. Our main result is that fiber-wise linear differential operators on $E$ are equivalent to (polynomial) derivations of an appropriate line bundle over $E^\ast$. We believe this might represent a first step towards a definition of multiplicative (resp.~infinitesimally multiplicative) differential operators on a Lie groupoid (resp.~a Lie algebroid). We also discuss the \emph{linearization} of a differential operator around a submanifold. 
\end{abstract}

\maketitle

\tableofcontents

\section{Introduction}
Given a vector bundle $E \to M$, it is often interesting to look at geometric structures (functions, vector fields, differential forms, etc.) on the total space $E$ that satisfy appropriate compatibility conditions with the vector bundle structure. Such compatibility is often referred to as \emph{linearity} in the literature. Accordingly, one speaks about \emph{linear functions}, \emph{linear vector fields}, \emph{linear differential forms}, etc., on the total space of a vector bundle. In the present paper, we will rather use the terminology ``\emph{fiber-wise linearity}'', to avoid confusion with other types of linearities. Now, the vector bundle structure is completely determined by (the smooth structure on $E$) and the action $h : \mathbb R \times E \to E$ of the monoid $(\mathbb R, \cdot)$ of multiplicative reals by fiber-wise scalar multiplication, $h(t, e) = te$, for all $t \in \mathbb R$, and all $e \in E$ (see, e.g.~\cite{GR2009}). It follows that the fiber-wise linearity of a geometric structure can be usually expressed purely in terms of $h$. For instance, a function $f$ on $E$ is fiber-wise linear if $h_t^\ast f = tf$ for all $t$. Similarly a vector field $X$ (resp.~a differential form $\omega$) on $E$ is fiber-wise linear if $h_t^\ast X = X$ for all $t \neq 0$ (resp.~$h_t^\ast \omega = t\omega$ for all $t$). A fiber-wise linear function is equivalent to a section of the dual vector bundle $E^\ast$, a fiber-wise linear vector fields is a section of the \emph{gauge algebroid} of $E$ (see, e.g, \cite{mackenzie}) and a fiber-wise linear differential $1$-form is equivalent to a section of the first jet bundle $J^1 E \to M$. The latter examples already show that fiber-wise linear structures on $E$ can encode interesting geometric structures on (vector bundles over) $M$. There are even more interesting examples. A fiber-wise linear symplectic structure $\omega$ on $E$ is equivalent to a vector bundle isomorphism $E \cong T^\ast M$. More precisely, there exists a unique fiber-wise linear $1$-form $\theta$ on $E$ such that $\omega = d \theta$, and there exists a unique vector bundle isomorphism $E \cong T^\ast M$ that identifies $\theta$ with the tautological $1$-form on $T^\ast M$, hence $\omega$ with the canonical symplectic structure on $T^\ast M$ (see, e.g., \cite{GR2009}, see also \cite{R2002}). There are more examples: a fiber-wise linear metric is equivalent to an isomorphism $E \cong T^\ast M$ together with a torsion free connection in $TM$ \cite{V20XX}. As a final remarkable example we recall that a fiber-wise linear Poisson structure on $E$ is the same as a Lie algebroid structure on $E^\ast$ \cite{mackenzie}. 

In this paper we propose the following definition of \emph{fiber-wise linear scalar differential operator}. An $\mathbb R$-linear differential operator $\Delta : C^\infty (E) \to C^\infty (E)$ of order $q$ is \emph{fiber-wise linear} if $h_t^\ast \Delta = t^{1-q} \Delta$ for all $t \neq 0$. This definition might seem weird at a first glance. However, it is supported by several different facts. For instance, according to our definition, a function and a vector field are fiber-wise linear if and only if they are fiber-wise linear when regarded as a $0$-th order and a first order scalar differential operator, respectively. Moreover the principal symbol of a fiber-wise linear differential operator is a fiber-wise linear symmetric multivector. Another supporting remark is that the Laplacian (acting on functions) of a fiber-wise linear metric is a fiber-wise linear differential operator. Finally, a scalar differential operator $\Delta$ can be \emph{linearized} around a submanifold producing a fiber-wise linear differential operator representing a first order approximation to $\Delta$ in the transverse direction with respect to the submanifold. All these facts suggest that our definition might indeed be the ``correct one''. Our main result is a description of fiber-wise linear differential operators in terms of somehow simpler data. More precisely we prove the following theorem (see Theorem \ref{theor:iso_DO_D_sym} for a more precise statement).
\begin{theorem}\label{theor:0}
Let $E \to M$ be a vector bundle. Then there is a degree inverting $C^\infty (M)$-linear bijection between fiber-wise linear scalar differential operators $\Delta : C^\infty (E) \to C^\infty (E)$ and polynomial derivations of the line bundle $E^\ast \times_M \wedge^{\mathrm{top}} E \to E^\ast$.
\end{theorem}
This theorem is a little surprising because it describes objects of higher order in derivatives (fiber-wise linear differential operators) in terms of objects of order $1$ in derivatives (derivations of an appropriate vector bundle). We hope that this result might be the starting point of a more thorough investigation of \emph{multiplicative differential operators} on Lie groupoids and, at the infinitesimal level, \emph{infinitesimally multiplicative differential operators} on Lie algebroids. Multiplicative (resp.~infinitesimally multiplicative) structures are geometric structures on a Lie groupoid (resp.~Lie algebroid) which are additionally compatible with the groupoid (resp.~algebroid) structure. In the last thirty years, starting from the pioneering works of Weinstein on symplectic groupoids \cite{W1987}, multiplicative structures captured the interest of a large community of people working in Poisson geometry and related fields, and today we have a precise description of several different multiplicative structures and their infinitesimal counterparts: infinitesimally multiplicative structures (see \cite{KS2016} for a survey). However, all the examples investigated so far are of order $1$ in derivatives and it would be interesting to investigate the compatibility of a Lie groupoid/algebroid with structures of higher order in derivatives, e.g.~higher order differential operators. This is a natural issue that might conjecturally lead to new important developments. As infinitesimally multiplicative structures are, in particular, fiber-wise linear structure, this paper might be also considered as a first step in this direction.

The paper is organized as follows. In Section \ref{sec:vector_fields} we recall what does it mean for a vector field on the total space of a vector bundle $E \to M$ to be \emph{fiber-wise linear}, i.e.~compatible with the vector bundle structure. In Section \ref{sec:multivectors} we discuss \emph{fiber-wise linear symmetric multivectors} and we describe them in terms of simpler data. This material is well known to experts (although it is scattered in the literature and it is hard to find a universal reference) and the first two sections are mainly intended to fix our notation. In Section \ref{sec:derivations} we recall what a derivation of a vector bundle $E$ is and introduce what we call $E$-multivectors, a ``derivation analogue'' of plain multivectors. We also discuss fiber-wise linear $E$-multivectors. These objects are not exactly of our primary interest but they play a very useful role in the proofs of our main theorems (Theorems \ref{theor:iso_DO_D_sym} and \ref{theor:linear}). To the best of our knowledge the material in the third section is mostly new. Section \ref{sec:DO} is an extremely compact introduction to linear differential operators on vector bundles, and, in particular, scalar differential operators. Section \ref{sec:FWL_DO} contains our main constructions and results: we define and study fiber-wise linear (scalar) differential operators on the total space of a vector bundle $E$. Somehow surprisingly, fiber-wise linear differential operators on $E$ form a transitive Lie-Rinehart algebra over fiber-wise polynomial functions on $E^\ast$, with abelian isotropies (Theorem \ref{theor:stabilizer}). The reason is ultimately explained by our main result, Theorem \ref{theor:0} above (see also Theorem \ref{theor:iso_DO_D_sym} below). As already announced, $E$-multivectors play a prominent role in the proof. In Section \ref{sec:linear} we discuss the \emph{linearization} of a scalar differential operator $\Delta$ around a submanifold $M$ in a larger manifold. The linearization of $\Delta$ is a fiber-wise linear differential operator on the total space of the normal bundle to $M$, and can be seen as a first order approximation to $\Delta$ in the direction transverse to $M$. The existence of a \emph{linearization construction} strongly supports our definition of fiber-wise linear differential operators.

\section{Core and Linear Vector Fields on a Vector Bundle}\label{sec:vector_fields}

As we mentioned in the introduction, the main aim of the paper is to explain what does it mean for a differential operator on the total space $E$ of a vector bundle $E \to M$ to be compatible with the vector bundle structure. We will reach our definition (Definition \ref{def:FWL_DO}) by stages. We first need to recall what does it mean for a function, a vector field and, more generally, a multivector on $E$, to be compatible with the vector bundle structure. We do this in the present and the next section. We adopt the general philosophy of \cite{GR2009, GR2011} where it is shown that a vector bundle structure is encoded in the fiber-wise scalar multiplication, and compatibility with the vector bundle structure is expressed in terms of such multiplication.

So, let $\pi : E \to M$ be a vector bundle. The fiber-wise scalar multiplication by a real number
\[
h : \mathbb{R} \times E \to E
\]
is an action of the multiplicative monoid of reals $(\mathbb R, \cdot)$. The algebra $C^\infty_{\mathrm{poly}} (E)$ of fiber-wise polynomial functions on the total space $E$ is non-negatively graded:
\[
C^\infty_{\mathrm{poly}} (E) = \bigoplus_{k = 0}^\infty C^\infty (E)_k,
\]  
and its $k$-th homogeneous piece $C^\infty (E)_k$ consists of homogeneous polynomial functions of degree $k$, i.e.~functions $f \in C^\infty (E)$ such that
\[
h_t^\ast (f) = t^k f
\]
for all $t \in \mathbb R$. Functions in $C^\infty (E)_0$ are just (pull-backs via the projection $\pi : E \to M$ of) functions on $M$. We call them \emph{core functions} and also denote them by $C_{\mathrm{core}}^\infty (E)$. They form a subalgebra in $C^\infty_{\mathrm{poly}} (E)$. Functions in $C^\infty (E)_1$ are fiber-wise linear (FWL for short in what follows) functions, and identify naturally with sections of the dual vector bundle $E^\ast$. We denote them by $C_{\mathrm{lin}}^\infty (E)$. They form a $C^\infty_{\mathrm{core}}(E)$ submodule in $C^\infty_{\mathrm{poly}} (E)$. We denote by $\ell_\varphi$ the linear function corresponding to the section $\varphi \in \Gamma (E^\ast)$. The terminology ``core function'' (and similarly ``core vector field'', etc., see below) is motivated by the theory of double vector bundles, where ``core sections'' are sections with an appropriate degree with respect to certain actions of $(\mathbb R, \cdot)$ (see, e.g.~\cite{mackenzie}).

A vector field $X \in \mathfrak X (E)$ on $E$ is \emph{fiber-wise polynomial}, or simply \emph{polynomial}, if it maps (fiber-wise) polynomial functions to polynomial functions. Polynomial vector fields $\mathfrak X_{\mathrm{poly}} (E)$ form a (graded) \emph{Lie-Rinehart algebra} over polynomial functions $C^\infty_{\mathrm{poly}} (E)$: 
\[
\mathfrak X_{\mathrm{poly}} (E) = \bigoplus_{k = -1}^\infty \mathfrak X (E)_k.
\]
We recall for later purposes that a Lie-Rinehart algebra over a commutative algebra $A$ is a vector space $L$, which is both an $A$-module and a Lie algebra acting on $A$ by derivations with the following two compatibilities:
\begin{itemize}
\item the Lie algebra action map $\rho : L \to \operatorname{Der} A$ is $A$-linear (it is often called the \emph{anchor}), and
\item the Lie bracket $[-,-] : L \times L \to L$ is a bi-derivation, i.e.~it satisfies the following Leibniz rule:
\[
[\lambda, a \mu] = \rho (\lambda) \mu + a [\lambda, \mu],\quad \lambda, \mu \in L, \quad a \in A.
\]
\end{itemize}
Lie-Rinehart algebras are purely algebraic counterparts of Lie algebroids. For more on Lie-Rinehart algebras, see, e.g.~\cite{H2004} and references therein.

Coming back to polynomial vector fields, the $k$-th homogeneous piece $\mathfrak X (E)_k$ of $\mathfrak X_{\mathrm{poly}} (E)$ consists of homogeneous ``polynomial'' vector fields of degree $k$, i.e.~vector fields $f \in \mathfrak X (E)$ such that
\[
h_t^\ast (X) = t^k X
\]
for all $t \neq 0$. Vector fields in $\mathfrak X (E)_{-1}$ are vertical lifts of sections of $E$. We call them \emph{core vector fields} and also denote them by $\mathfrak X_{\mathrm{core}} (E)$. They form an abelian Lie-Rinehart subalgebra (beware over the subalgebra $C^\infty_{\mathrm{core}} (E) = C^\infty (M)$) in $\mathfrak X_{\mathrm{poly}} (E)$. We denote by $e^\uparrow$ the vertical lift of a section $e \in \Gamma(E)$.

Vector fields in $\mathfrak X (E)_0$ are, by definition, fiber-wise linear (FWL) vector fields. They can be equivalently characterized as vector fields preserving linear functions and they satisfy the following property
\[
[X, Y] \in \mathfrak X_{\mathrm{core}} (E), \quad \text{for all $Y \in \mathfrak X_{\mathrm{core}} (E)$}.
\]
We denote FWL vector fields by $\mathfrak X_{\mathrm{lin}} (E)$. They form a Lie-Rinehart subalgebra (over $C^\infty_{\mathrm{core}}(E)$) in $\mathfrak X_{\mathrm{poly}} (E)$.

If $(x^i, u^\alpha)$ are vector bundle coordinates, then a function $f \in C^\infty (E)$ is a core function if and only if $f = f(x)$ and it is a linear functions if and only if, locally, $f = f_\alpha (x) u^\alpha$. Similarly, a vector field $X \in \mathfrak X (E)$ is a core vector field if and only if, locally, 
\[
X = X^\alpha (x) \frac{\partial}{\partial u^\alpha}
\]
 and it is a linear vector field if and only if, locally,
 \[
 X = X^i (x) \frac{\partial}{\partial x^i} + X^\alpha_\beta (x) u^\beta \frac{\partial}{\partial u^\alpha}.
 \]
 
 \begin{remark}\label{rem:linear_tensor}
 There is also a useful notion of \emph{FWL tensor} on $E$. Let $\mathcal T \in \Gamma (T^{\otimes r} E \otimes T^\ast{}^{\otimes s}E)$ be a tensor field of type $(r,s)$. Then $\mathcal T$ is a \emph{FWL tensor} if
 \[
 h_t^\ast \mathcal T = t^{1-r} \mathcal T
 \] 
 for all $t \neq 0$. Notice that FWL tensors are called \emph{linear tensor fields} in \cite{BD2019}, where they are characterized in terms of the fiber-wise addition in $E$ (rather than via the fiber-wise multiplication $h$ as we do). As an instance, consider a metric $g \in \Gamma (S^2 T^\ast E)$. It is linear if $h_t^\ast g = t g$ for all $t$, and it is easy to see that this is in turn equivalent to $g$ being locally of the form
 \[
 g =  g_{\alpha i} (x) du^\alpha \odot dx^i + g_{\alpha|ij} (x) u^\alpha dx^i \odot dx^j.
 \]
 Notice that the non-degeneracy condition then implies that the $x$-dependent matrix $\left( g_{\alpha i}\right)$ is invertible. In particular, the dimension of $M$ and the rank of $E$ must agree. Even more, denoting by $T^\pi E = \ker d\pi$ the $\pi$-vertical bundle, the composition
 \begin{equation}\label{eq:comp}
E \overset{\cong}{\to} {T^\pi E|_M} \hookrightarrow T E|_M \overset{\flat}{\to} T^\ast E|_M \to T^\ast M
 \end{equation}
 is a vector bundle isomorphism. Here $\flat : TE \to T^\ast E$ is the musical isomorphism, the second and the fourth arrow are those induced by the canonical direct sum decomposition, $TE|_M = TM \oplus T^\pi E|_M$, and the first arrow is the canonical isomorphism. In other words, a non-degenerate symmetric covariant $2$-tensor $g$ can only exist on the total space of (a vector bundle isomorphic to) the cotangent bundle. Finally, in standard coordinates $(x^i, p_i)$ on $T^\ast M$, $g$ looks like
 \begin{equation}\label{eq:g_linear}
 g = dp_i \odot dx^i  - \Gamma_{ij}^k(x)p_k dx^i \odot dx^j,
 \end{equation}
for some appropriate local functions $\Gamma_{ij}^k (x)$. In particular, $g$ is necessarily of split signature. For more on FWL metrics see \cite{V20XX}.

 \end{remark}
 
 \section{More on FWL Multivector Fields}\label{sec:multivectors}
 
 The material in this section is well-known to experts, and it is partly folklore, partly scattered in the literature. For this reason it is hard to give precise references (the reader may consult, e.g., \cite[Appendix A]{LV2019} and references therein, although that reference does not cover the same exact material as the following one). In any case, most of the proofs are straightforward and we omit them.
 
We will need to consider FWL symmetric multivectors. According to Remark \ref{rem:linear_tensor}, a $k$-multivector $P$ on the total space $E$ of a vector bundle $E \to M$ is \emph{FWL} if
 \[
 h_t^\ast (P) = t^{1-k} P
 \]
 for all $t \neq 0$. We denote by $\mathfrak X^\bullet_{\mathrm{sym}, \mathrm{lin}}(E)$ 
 FWL symmetric 
 multivectors.
 
 There is a useful characterization of a FWL $k$-multivectors. Namely, a $k$-multivector $P$ on $E$ is FWL if and only if
\begin{enumerate}
\item $P (f_1, \ldots, f_{k}) \in C^\infty_{\mathrm{lin}}(E)$,
\item $P (f_1, \ldots, f_{k-1}, h_1) \in C^\infty_{\mathrm{core}}(E)$,
\item $P(f_1, \ldots, f_{k-2}, h_1, h_2) = 0$,
\end{enumerate}
for all $f_i \in C^\infty_{\mathrm{lin}}(E)$, and all $h_j \in C^\infty_{\mathrm{core}}(E)$. In particular, a FWL symmetric $k$-multivector $P$ determines a pair of maps $(D_P, l_P)$:
\[
D_P : \underset{\text{$k$ times}}{\underbrace{\Gamma (E^\ast) \times \cdots \times \Gamma (E^\ast)}} \to \Gamma (E^\ast)
\]
and
\[
l_P : \underset{\text{$k-1$ times}}{\underbrace{\Gamma (E^\ast) \times \cdots \times \Gamma (E^\ast)}} \times C^\infty (M) \to C^\infty (M)
\]
via 
\[
\begin{aligned}
\ell_{D_P (\varphi_1, \ldots, \varphi_k)} & = P (\ell_{\varphi_1}, \ldots, \ell_{\varphi_k})\\
l_P (\varphi_1, \ldots, \varphi_{k-1}, f) & = P (\ell_{\varphi_1}, \ldots, \ell_{\varphi_k}, f)
\end{aligned}
\]
for all $\varphi_i \in \Gamma (E^\ast)$, and all $f \in C^\infty (M)$. The maps $D_P, l_P$ satisfy the following properties:
\begin{enumerate}
\item $D_P$ is $\mathbb R$-multilinear and symmetric,
\item $l_P$ is $C^\infty (M)$-multilinear and symmetric in the first $(k-1)$-arguments,
\item $D_P (\varphi_1, \ldots, \varphi_{k-1}, f \varphi_k) = f D_P (\varphi_1, \ldots,\varphi_k) + l_P (\varphi_1, \ldots, \varphi_{k-1}, f) \varphi_k$, for $\varphi_i \in \Gamma (E^\ast)$, and $f \in C^\infty (M)$,
\item $l_P$ is a derivation in its last argument.
\end{enumerate}
In particular, $l_P$ can be seen as a vector bundle map $l_P : S^{k-1}E^\ast \to TM$, and we will often write 
\[
l_P (\varphi_1, \ldots, \varphi_{k-1}) (f)
\]
instead of $l_P (\varphi_1, \ldots, \varphi_{k-1}, f)$. The assignment $P \mapsto (D_P, l_P)$ establishes a $C^\infty (M)$-linear bijection between FWL symmetric $k$-multivectors on $E$ and \emph{$k$-multiderivations} of $E^\ast$, i.e.~pairs $(D,l)$ consisting of a map $D : \Gamma (E^\ast) \times \cdots \times \Gamma (E^\ast) \to \Gamma (E^\ast)$ and a vector bundle map $l : S^{k-1} E^\ast \to TM$ (equivalently a section of $S^{k-1}E \otimes TM$) satisfying 
\[
D (\varphi_1, \ldots, \varphi_{k-1}, f \varphi_k) = f D (\varphi_1, \ldots,\varphi_k) + l (\varphi_1, \ldots, \varphi_{k-1}) (f) \varphi_k.
\]
The map $l$ is sometimes called the \emph{symbol} of $D$ and it is completely determined by $D$. For this reason, we will often refer to $D$ itself as a $k$-multiderivation (see, e.g., \cite{GG2003, CM2008} for a skew-symmetric version of multiderivations). 

Recall that there is a natural Poisson bracket $\{-,-\}$ on symmetric multivectors given by the following (Gerstenhaber-type) formula
\begin{equation}\label{eq:Poisson_multiv}
\begin{aligned}
& \{P_1, P_2\} (f_1, \ldots, f_{k_1 + k_2 +1}) \\
& = \sum_{\sigma \in S_{k_2+1, k_1}}P_1 \left( P_2 \left(f_{\sigma(1)}, \ldots, f_{\sigma(k_2+1)}\right), f_{\sigma(k_2+2)}, \ldots, f_{\sigma(k_1 + k_2 +1)}\right) \\
& \quad - \sum_{\sigma \in S_{k_1+1, k_2}}P_2 \left( P_1 \left(f_{\sigma(1)}, \ldots, f_{\sigma(k_1+1)}\right), f_{\sigma(k_1+2)}, \ldots, f_{\sigma(k_1 + k_2 +1)}\right)
\end{aligned}
\end{equation}
for all $(k_1+1)$-multivectors $P_1$, $(k_2+1)$-multivectors $P_2$, and all functions $f_i$, where $S_{k, h}$ denotes $(k, h)$-unshuffles. The Poisson bracket (\ref{eq:Poisson_multiv}) preserves FWL symmetric multivectors and the Poisson bracket $\{P_1, P_2\}$ of the FWL Poisson multivectors $P_1, P_2 \in \mathfrak X^{\bullet}_{\mathrm{sym}, \mathrm{lin}}(E)$ identifies with the obvious \emph{Gerstenhaber-like} bracket $\{ D_1, D_2\}$ of the associated multiderivations $D_1,D_2$.

FWL symmetric multivectors on $E$ do also identify with polynomial vector fields on $E^\ast$. To see this, it is useful to talk about \emph{core multivectors} first. A $k$-multivector $P$ on $E$ is \emph{core} if
\[
h_t^\ast (P) = t^{-k} P
\] 
for all $t \neq 0$. Core $k$-multivectors can be characterized as those multivectors $P$ such that 
\begin{enumerate}
\item $P (f_1, \ldots, f_{k}) \in C^\infty_{\mathrm{core}}(E) = C^\infty (M)$,
\item $P (f_1, \ldots, f_{k}, h) = 0$,
\end{enumerate}
for all $f_i \in C^\infty_{\mathrm{lin}} (E)$ and $h \in C^\infty_{\mathrm{core}} (E)$, and they form a subalgebra $\mathfrak X^\bullet_{\mathrm{sym}, \mathrm{core}}(E)$ in the associative, commutative algebra $\mathfrak X^\bullet_{\mathrm{sym}}(E)$ (with the symmetric product). More precisely, $\mathfrak X^\bullet_{\mathrm{sym}, \mathrm{core}}(E)$ is the subalgebra spanned by core functions and core vector fields. In particular, $\mathfrak X^\bullet_{\mathrm{sym}, \mathrm{core}}(E)$ identifies with sections $\Gamma (S^\bullet E)$ of the symmetric algebra of $E$ via
\[
(e_1)^\uparrow \odot \cdots \odot (e_k)^\uparrow \mapsto e_1 \odot \cdots \odot e_k.
\]
$e_i \in \Gamma (E)$. In its turn, $\Gamma (S^\bullet E)$ identifies with polynomial functions on $E^\ast$ in the via the (degree preserving) algebra isomorphism
\[
\Gamma (S^\bullet E) \to C^\infty_{\mathrm{poly}} (E^\ast), \quad e_1 \odot \cdots \odot e_q \mapsto \ell_{e_1} \cdots \ell_{e_q},
\]
and, in what follows, we will often understand the latter identifications. Notice that the resulting isomorphism
\[
\mathfrak X^\bullet_{\mathrm{sym}, \mathrm{core}}(E) \to C^\infty_{\mathrm{poly}}(E^\ast), \quad P \mapsto F_P
\]
is given by
\[
F_P (\varphi_x) = \frac{1}{k!}P (\ell_\varphi, \ldots, \ell_\varphi)(x)
\]
$P \in \mathfrak X^k_{\mathrm{sym}, \mathrm{core}}(E)$, $\varphi \in \Gamma (E^\ast)$ and $x \in M$.

We can now go back to FWL symmetric multivectors. The symmetric product of a core symmetric multivector and a FWL one is a FWL multivector, and this turns $\mathfrak X^\bullet_{\mathrm{sym},\mathrm{lin}}(E)$ into a $\Gamma (S^\bullet E)$-module. Now, let $P \in \mathfrak X^\bullet_{\mathrm{sym}, \mathrm{lin}}(E)$. It is easy to see that the Poisson bracket
$H_P = \{P, - \}$ preserves core multivectors. Hence, it is a derivation of the commutative algebra $\mathfrak X^\bullet_{\mathrm{sym}, \mathrm{core}}(E) \cong C^\infty_{\mathrm{poly}} (E^\ast)$. In its turn, $H_P$ extends uniquely to a polynomial vector field, also denoted $H_P$, on $E^\ast$. The assignment $P \mapsto H_P$ establishes a degree inverting isomorphism of Lie algebras, between the Lie algebra of linear symmetric multivectors on $E$ (with the Poisson bracket) and polynomial vector fields on $E^\ast$ (with the commutator). When we equip $\mathfrak X^\bullet_{\mathrm{sym}, \mathrm{lin}}(E)$ with the symmetric product by a core multivector, the latter isomorphism becomes an isomorphism of Lie-Rinehart algebras.

Finally, we remark that linear symmetric multivectors fit in the following short exact sequence
\begin{equation}\label{eq:ses_X_sym_lin}
0 \longrightarrow \Gamma (S^\bullet E \otimes E^\ast) \longrightarrow \mathfrak X^\bullet_{\mathrm{sym}, \mathrm{lin}} (E) \overset{l}{\longrightarrow} \Gamma (S^{\bullet-1} E \otimes TM)\longrightarrow 0
\end{equation}
where the second arrow identifies the section $e_1 \odot \cdots \odot e_k \otimes \varphi$ of $S^k E \otimes E^\ast$ with the FWL $k$-multivector field
$
\ell_{\varphi} e_1^\uparrow \odot \cdots \odot e_k^\uparrow.
$

\section{More on Derivations of a Vector Bundle}\label{sec:derivations}

In this section, for a vector bundle $V \to M$, we introduce a notion of (symmetric) \emph{$V$-multivector} (Definition \ref{def:V-multiv}). To the best of our knowledge this notion is new. It will play a significant role in the description of fiber-wise linear differential operators provided in Section \ref{sec:FWL_DO}. Symmetric $V$-multivectors are in many respect similar to plain symmetric multivectors, so the proofs of most of the statements in this section parallel the proofs of the analogous statements for multivectors and we omit them.

We begin with a vector bundle $E \to M$ and remark that the space $\mathfrak X^1_{\mathrm{lin}}(E) = \mathfrak X_{\mathrm{lin}}(E)$ of linear vector fields is of particular interest. The assignment $X \mapsto D_X$ establishes an isomorphism of Lie-Rinehart algebras (over $C^\infty (M)$) between linear vector fields on $E$ and \emph{derivations} of $E^\ast$, i.e.~$1$-multiderivations. We stress that
\[
X (\ell_\varphi) = \ell_{D_X \varphi}, \quad \varphi \in \Gamma (E^\ast).
\]

In the following, we denote by $\mathfrak D (V)$ the Lie-Rinehart algebra of derivations of a vector bundle $V$. It is the Lie-Rinehart algebra of sections of a Lie algebroid $DV \to M$ whose Lie bracket is the commutator of derivations and whose anchor is the symbol map $D \mapsto l_D$. 

Notice also that the assignment $X \mapsto H_X$ (see the last paragraph of the previous section) does also establish an isomorphism of Lie-Rinehart algebras (over $C^\infty (M)$) between linear vector fields on $E$ and linear vector fields on $E^\ast$.
Accordingly we have a canonical Lie algebroid isomorphism $DE \mapsto DE^\ast$, $D \mapsto D^\ast$ which is explicitly given by
\[
\langle D^\ast \varphi, e \rangle = l_D \left(\langle \varphi, e \rangle \right) - \langle \varphi, D e \rangle,
\]
for every $\varphi \in \Gamma (E^\ast)$, $e \in\Gamma (E)$, where $\langle -,-\rangle : E^\ast \otimes E \to \mathbb R_M := M \times \mathbb R$ is the duality pairing. In the following we will simply denote by $D$ the derivation of $E^\ast$ induced by a derivation of $E$ (and vice-versa). It is easy to see that
\[
[X, e^\uparrow] = (D_X e)^\uparrow, \quad e \in \Gamma (E).
\]

More generally, a derivation $D$ of a vector bundle $V$ induces a derivation, also denoted $D$, in each component of the whole (symmetric, resp.~alternating) tensor algebra of $V \oplus V^\ast$. The latter derivation is defined imposing the obvious Leibniz rule with respect to the tensor product and the contraction by an element in the dual.

We are now ready to define the algebra of \emph{$V$-multivectors}, which is a ``derivation analogue'' of the Poisson algebra
 of symmetric multivectors. So, let $V \to M$ be a vector
  bundle, and consider the graded space $\tilde{\mathfrak D}{}^{\bullet} := \mathfrak
   X_{\mathrm{sym}}^{\bullet-1} (M) \otimes \mathfrak D (V)$, where the tensor 
   product is over functions on $M$. Consider the graded subspace $\mathfrak
    D^{\bullet}_{\mathrm{sym}} (V) \subset \tilde{\mathfrak D}{}^\bullet$ consisting of
     elements projecting on symmetric multivectors $\mathfrak X_{\mathrm{sym}}
     ^{\bullet} (M) \hookrightarrow \mathfrak X_{\mathrm{sym}}^{\bullet -1} (M) \otimes
      \mathfrak X (M)$ via
\[
\mathrm{id} \otimes l : \tilde{\mathfrak D}{}^\bullet \to \mathfrak X_{\mathrm{sym}}^\bullet (M) \otimes \mathfrak X (M).
\]
We denote by
\[
L : \mathfrak D_{\mathrm{sym}}^{\bullet}(V) \to \mathfrak X_{\mathrm{sym}}^\bullet (M)
\]
the projection. Notice that $\mathfrak D_{\mathrm{sym}}^{\bullet}(V)$ fits in an exact sequence:
\begin{equation}\label{eq:ses_V-derivations}
0 \longrightarrow \mathfrak X^{\bullet - 1}_{\mathrm{sym}} (M) \otimes \Gamma (\operatorname{End} V) \longrightarrow \mathfrak D_{\mathrm{sym}}^{\bullet}(V )\overset{L}{\longrightarrow} \mathfrak X^{\bullet }_{\mathrm{sym}} (M) \longrightarrow 0.
\end{equation}
\begin{definition}\label{def:V-multiv}
Elements in $\mathfrak D_{\mathrm{sym}}^{k}(V)$ are \emph{symmetric $k$-$V$-multivectors}.
\end{definition} 
A symmetric $k$-$V$-multivector $D \in \mathfrak D_{\mathrm{sym}}^{k}(V)$ will be often interpreted as an operator
\[
D : C^\infty (M) \times \cdots C^\infty (M) \times \Gamma (V) \to \Gamma (V), \quad (f_1, \ldots, f_{k-1}, v) \mapsto D(f_1, \ldots, f_{k-1}|v).
\]

%

\begin{lemma}\label{lem:Poisson_D}
The space $\mathfrak D^{\bullet}_{\mathrm{sym}}(V)$ of symmetric $V$-multivectors is a Poisson algebra when equipped with 
\begin{enumerate}
\item the associative product given by
\[
\begin{aligned}
& D_1 \cdot D_2 (f_1, \ldots, f_{k_1+ k_2 + 1}|v) \\
& = \sum_{\sigma \in S_{k_1 + 1, k_2}}L_{D_1}(f_{\sigma(1)}, \ldots, f_{\sigma(k_1 + 1)}) D_2(f_{\sigma(k_1+2)}, \ldots, f_{\sigma(k_1 + k_2 +1)}|v) \\
& \quad + 
\sum_{\sigma \in S_{k_2 + 1, k_1 }}L_{D_2}(f_{\sigma(1)}, \ldots, f_{\sigma(k_2 + 1)}) D_1(f_{\sigma(k_2+2)}, \ldots, f_{\sigma(k_1 + k_2 +1)}|v)
\end{aligned}
\]
for all $f_i \in C^\infty (M)$, $v \in \Gamma (V)$; and
\item the Lie bracket $\{-,-\}$ given by
\begin{equation}\label{eq:Poisson_V-multivectors}
\{D_1, D_2\} = D_1 \bullet D_2 -  D_2 \bullet D_1 
\end{equation}
where
\[
D_1 \bullet D_2 : C^\infty (M) \times \cdots \times C^\infty (M) \times \Gamma (V) \to \Gamma(V)
\]
is the operator given by
\[
\begin{aligned}
& D_1 \bullet D_2 \left( f_1, \ldots, f_{k_1 + k_2} |v  \right) \\
& =
\sum_{\sigma \in S_{k_1,k_2}} D_1 \left(f_{\sigma(1)}, \ldots, f_{\sigma(k_1)} | D_2 (f_{\sigma (k_1 + 1)}, \ldots, f_{\sigma (k_1 + k_2)}| v) \right) \\
& \quad + \sum_{\sigma \in S_{k_1-1,k_2 + 1}}D_1 \left(f_{\sigma(1)}, \ldots, f_{\sigma(k_1 - 1)}, L_{D_2} (f_{\sigma (k_1)}, \ldots, f_{\sigma (k_1 + k_2)})|v \right)
\end{aligned}
\]
for all $f_i \in C^\infty (M)$, and $v \in \Gamma (V)$.\end{enumerate}
Here $D_1 \in \mathfrak D_{\mathrm{sym}}^{k_1 + 1}(V)$ and $D_2 \in \mathfrak D_{\mathrm{sym}}^{k_2 + 1}(V)$, $D_1 \cdot D_2 \in \mathfrak D_{\mathrm{sym}}^{k_1 + k_2 + 2}(V)$, and $\{D_1,  D_2\} \in \mathfrak D_{\mathrm{sym}}^{k_1 + k_2 + 1}(V)$.

The map
\[
L : \mathfrak D_{\mathrm{sym}}^{\bullet}(V) \to \mathfrak X_{\mathrm{sym}}^\bullet (M)
\]
is a surjective Poisson algebra map.
\end{lemma}

\begin{proof}
A long and tedious computation that we omit.
\end{proof}

\begin{remark}
When $V$ is a line bundle, the map $\mathfrak X^{\bullet - 1}_{\mathrm{sym}} (M) \otimes \Gamma (\operatorname{End} V) \to \mathfrak D_{\mathrm{sym}}^{\bullet}(V )$ embeds $\mathfrak X^{\bullet - 1}_{\mathrm{sym}} (M) = \mathfrak X^{\bullet - 1}_{\mathrm{sym}} (M) \otimes C^\infty (M) = \mathfrak X^{\bullet - 1}_{\mathrm{sym}} (M) \otimes \Gamma (\operatorname{End} V) $ into $\mathfrak D_{\mathrm{sym}}^{\bullet}(V)$ as an abelian subalgebra and an ideal.
\end{remark}

There is also a notion of FWL $V$-multivector. In order to discuss it, it is is useful to discuss derivations of pull-back vector bundles first. So, let $V$ be a vector bundle, and consider its pull-back $V_{\mathcal P} := \pi^\ast V$ along a surjective submersion $\pi: \mathcal P \to M$. Clearly, a derivation $D$ of $V_{\mathcal P}$ is completely determined by its symbol and its action on pull-back sections. The restriction $D_M := D|_{\Gamma (V)}$ of $D$ to pull-back sections is a \emph{derivation along $\pi$}, i.e.~it is an $\mathbb R$-linear map $D_M : \Gamma (V) \to \Gamma (V_{\mathcal P})$ and there exists a, necessarily unique, vector field along $\pi$, denoted $l_{D_M} \in \Gamma (\pi^\ast TM)$, fitting in the Leibniz rule
\[
D_M (fv) = \pi^\ast (f) D_M (v) + l_{D_M}(f) v, \quad f \in C^\infty (M), \quad v \in \Gamma (V).
\]
The correspondence $D \mapsto (l_D, D_M)$ establishes a $C^\infty(M)$-linear bijection between derivations $D$ of $V_{\mathcal P}$ and pairs $(X, D_M)$ consisting of a vector field $X \in \mathfrak X (\mathcal P)$ and a derivation along $\pi$ satisfying the following additional compatibility: $d \pi \circ X = l_{D_M}$. When $\mathcal P = E \to M$ is a vector bundle, it makes sense to talk about \emph{polynomial sections of $V_E$}. Namely, $\Gamma (V_E) = C^\infty (E) \otimes \Gamma (V)$,  where the tensor product is over $C^\infty (M)$, and multiplicative reals act on $\Gamma (V_E)$ via their action on the first factor. As for functions, we denote by $h^\ast$ this action. A section $v$ of $\Gamma (V_E)$ is polynomial of degree $k$ if $h_t^\ast (v) = t^k v$ for all $t \in \mathbb R$, in other words $v \in C^\infty_{\mathrm{poly}}(E) \otimes \Gamma (V)$. Denote by $\Gamma (V_E)_k$ the space of polynomial sections of degree $k$, and by
\[
\Gamma_{\mathrm{poly}} (V_E) := \bigoplus_{k= 0}^\infty \Gamma (V_E)_k 
\]
the space of all polynomial sections. \emph{FWL sections} are degree $1$ sections and they identify with sections of $E^\ast \otimes V$. \emph{Core sections} are degree $0$ sections and they identify simply with sections of $V$. Similarly to vector fields, a derivation $D$ of $V_{E}$ is \emph{polynomial} of degree $k$ if it maps polynomial sections of degree $h$ to polynomial sections of degree $k + h$, in other words
\[
h_t^\ast D = t^k D
\]
for all $t \neq 0$. Polynomial derivations of $V_{E}$ will be denoted $\mathfrak D_{\mathrm{poly}}(V_{E})$ and they correspond to pairs $(X, D_M)$ where $X \in \mathfrak X_{\mathrm{poly}}(E)$ and $D_M$ takes values in polynomial sections.

We can also consider polynomial symmetric $V_E$-multivectors. We will only need core and FWL ones. A symmetric $V_E$-$k$-multivector $D$ is \emph{FWL} (resp.~\emph{core}) if it is polynomial of degree $1-k$ (resp.~$-k$), i.e.
\[
h_t^\ast D = t^{1-k} D\quad \text{(resp.~$h_t^\ast D = t^{-k} D$)}
\]
for all $t \neq 0$. We denote by $\mathfrak D^\bullet_{\mathrm{sym}, \mathrm{lin}} (V_E)$ (resp.~$\mathfrak D^\bullet_{\mathrm{sym}, \mathrm{core}} (V_E))$ the space of FWL (resp.~core) symmetric $V_E$-multivectors. The Lie bracket $\{-,-\}$ on $V$-multivectors preserves FWL ones. Additionally, the projection $L : \mathfrak D^\bullet_{\mathrm{sym}} (V_E) \to \mathfrak X^\bullet_{\mathrm{sym}}(E)$ maps FWL $V$-multivectors to FWL multivectors and we get a short exact sequence of Lie algebras:
\begin{equation}\label{eq:ses_D_sym_lin}
0 \longrightarrow \mathfrak X^{\bullet - 1}_{\mathrm{sym}, \mathrm{lin}} (E) \longrightarrow \mathfrak D^\bullet_{\mathrm{sym}, \mathrm{lin}} (V_E) \overset{L}{\longrightarrow} \mathfrak X^\bullet_{\mathrm{sym}, \mathrm{lin}} (E) \longrightarrow 0.
\end{equation}

\begin{proposition}
A $k$-$V_E$-multivector $D$ is core if and only if
\begin{enumerate}
\item $D (f_1, \ldots, f_{k-1}| v) \in \Gamma_{\mathrm{core}}(V_E)$,
\item $D (f_1, \ldots, f_{k-1}| w) = 0$,
\item $D (f_1, \ldots, f_{k-2}, h | v) = 0$,
\end{enumerate}
for all $f_i \in C^\infty_{\mathrm{lin}} (E)$, $h \in C^\infty_{\mathrm{core}}(E)$, $v \in \Gamma_{\mathrm{lin}}(V_E) := \Gamma (V_E)_1$, and all $w \in \Gamma_{\mathrm{core}}(V_E):= \Gamma (V_E)_0 = \Gamma (V)$.
\end{proposition}

\begin{proof}
Straightforward.
\end{proof}

It easily follows from the above proposition that a core $k$-$V_E$-multivector is completely determined by its symbol. More precisely, the symbol map $D \mapsto L_D$ establishes a one-to-one correspondence between core $k$-$V_E$-multivectors and core multivectors. We conclude that $\mathfrak D^\bullet_{\mathrm{sym}, \mathrm{core}}(V_E) \cong \mathfrak X^\bullet_{\mathrm{sym}, \mathrm{core}}(E) \cong C^\infty_{\mathrm{poly}} (E^\ast)$. 

\begin{proposition}
A symmetric $k$-$V_E$-multivector $D$ is FWL if and only if
\begin{enumerate}
\item $D (f_1, \ldots, f_{k-1}| v) \in \Gamma_{\mathrm{lin}}(V_E)$,
\item $D (f_1, \ldots, f_{k-1}| w) \in \Gamma_{\mathrm{core}}(V_E)$,
\item $D (f_1, \ldots, f_{k-2}, h_1 | v) \in \Gamma_{\mathrm{core}}(V_E)$,
\item $D(f_1, \ldots, f_{k-2}, h_{1} | w) = 0$,
\item $D(f_1, \ldots, f_{k-3}, h_{1}, h_2 | v) = 0$,
\end{enumerate}
for all $f_i \in C^\infty_{\mathrm{lin}} (E)$, $h_j \in C^\infty_{\mathrm{core}}(E)$, $v \in \Gamma_{\mathrm{lin}}(V_E) := \Gamma (V_E)_1$, and all $w \in \Gamma_{\mathrm{core}}(V_E):= \Gamma (V_E)_0 = \Gamma (V)$.
\end{proposition}
\begin{proof}
Straightforward.
\end{proof}
 In particular, a FWL symmetric $k$-$V$-multivector $D$ determines a map:
\[
\Phi_D : \underset{\text{$k-1$ times}}{\underbrace{\Gamma (E^\ast) \times \cdots \times \Gamma (E^\ast)}} \to  \mathfrak D (V)
\]
via 
\[
\Phi_D (\varphi_1, \ldots, \varphi_{k-1})(w)  = D (\ell_{\varphi_1}, \ldots, \ell_{\varphi_k} | w)
\]
for all $\varphi_i \in \Gamma (E^\ast)$, and all $w \in \Gamma (V)$. The map $\Phi_D$ is $C^\infty (M)$-multilinear and symmetric. Hence, it can be seen as a vector bundle map $\Phi_D : S^{k-1}E^\ast \to DV$, or, equivalently, as a section of $S^{k-1} E \otimes DV$.

\begin{proposition}\label{prop:V_E-multi_pairs}
The assignment $D \mapsto (L_D, \Phi_D)$ establishes a $C^\infty (M)$-linear bijections between FWL symmetric $V_E$-multivectors $D \in \mathfrak D^\bullet_{\mathrm{sym}, \mathrm{lin}} (V_E)$ and pairs $(P, \Phi)$ consisting of a FWL symmetric multivector $P \in \mathfrak X^\bullet_{\mathrm{sym}, \mathrm{lin}} (E)$ and a vector bundle map $\Phi : S^{k-1}E^\ast \to DV$ such that $l_P = l \circ \Phi$.
\end{proposition}
\begin{proof}
Easy and left to the reader.
\end{proof}

According to Proposition \ref{prop:V_E-multi_pairs} we will sometimes call the pair $(L_D, \Phi_D)$ itself a FWL symmetric $V_E$-multivector.

Now, we can combine the exact sequences (\ref{eq:ses_X_sym_lin}) and (\ref{eq:ses_D_sym_lin}) in one exact commutative diagram:
\[
\begin{array}{c}
\xymatrix{ & 0 & 0 & 0 & \\
0 \ar[r] & \Gamma (S^{\bullet -1} E \otimes \operatorname{End} V) \ar[u] \ar[r] & \Gamma (S^{\bullet-1} E \otimes DV) \ar[r]^-{\operatorname{id} \otimes l} \ar[u] & \Gamma (S^{\bullet-1} E \otimes TM) \ar[r] \ar[u] & 0 \\
0 \ar[r] & \Gamma (S^{\bullet -1} E \otimes \operatorname{End} V) \ar@{=}[u] \ar[r] & \mathfrak D^\bullet_{\mathrm{sym}, \mathrm{lin}}(V_E) \ar[u]^\Phi \ar[r]^-L & \mathfrak X_{\mathrm{sym}, \mathrm{lin}}^{\bullet}(E) \ar[u]^l \ar[r] & 0 \\
& 0 \ar[u] \ar[r] & \Gamma (S^{\bullet}E \otimes E^\ast) \ar[u]^-I \ar@{=}[r]
& \Gamma (S^{\bullet}E \otimes E^\ast) \ar[u] \ar[r] & 0 \\
& & 0 \ar[u] & 0 \ar[u] &
}
\end{array}.
\]
We only need to explain the map $I$. To do that, we first remark that $\pi$-vertical vector fields act naturally on sections of $V_E$, via
\[
X (f \otimes v) = X(f) \otimes v
\]
for all $X \in \Gamma (T^\pi E)$, $f \in C^\infty (E)$, and $v \in \Gamma(V)$. Now, take $e_1, \ldots, e_k \in \Gamma (E)$ and $\varphi \in \Gamma (E^\ast)$. Then
\[
I (e_1 \odot \cdots \odot e_k \otimes \varphi)(f_1, \ldots, f_{k-1} | v) = \frac{1}{k!} \cdot \ell_\varphi  \sum_{\sigma \in S_{k}} e_{\sigma(1)}^\uparrow (f_1) \cdots e_{\sigma(k-1)}^\uparrow (f_{k-1}) e_{\sigma(k)}^\uparrow (v).
\]
Equivalently, we can interpret $e_1 \odot \cdots \odot e_k$ as a core $V_E$-multivector, via the isomorphism $D^\bullet_{\mathrm{sym}, \mathrm{core}}(V_E) \cong \Gamma (S^\bullet E)$, and then multiply by the FWL function $\ell_\varphi$ to get a FWL $V_E$-multivector. 

We conclude this section showing that FWL symmetric $V_E$-multivectors do also identify with polynomial derivations of $V_{E^\ast}$. This is an easy consequence (among other things) of Proposition \ref{prop:V_E-multi_pairs}. Indeed, take $D \in \mathfrak D^\bullet_{\mathrm{sym}, \mathrm{lin}}(V_E)$ and let $(L_D, \Phi_D)$ be the corresponding pair. Denote by $\pi : E^\ast \to M$ the projection. We claim that $\Phi_D$ can be seen as a derivation along $\pi$. Indeed $\Phi_D$ is a section of $S^{\bullet -1} E \otimes DL$ and, by acting on a section $v \in \Gamma (V)$ with the $DL$-factor, we get a section $\Phi_D(v)$ of $S^{\bullet -1} E \otimes V$, i.e.~a polynomial section of $\Gamma (V_{E^\ast})$. In the following, we use this construction to interpret $\Phi_D$ as a derivation along $\pi$. If we do so, the pair $(H_{L_D}, \Phi_D)$ consists of a vector field on $E^\ast$, and a derivation $\Phi_D$ along $\pi$, with the additional property that $d \pi \circ H_{L_D} = l \circ \Phi_D$, hence it corresponds to a (polynomial) derivation $D^\ast$ of the pull-back vector bundle $V_{E^\ast}$. Finally a tedious, but straightforward computation shows that the bijection $D \mapsto D^\ast$ between linear $V_E$-multivectors and polynomial derivations of $V_{E^\ast}$ obtained in this way do also preserve the Lie algebra structures. When we equip $\mathfrak D^\bullet_{\mathrm{sym}, \mathrm{lin}}(E)$ with the product by a core $V_E$-multivector, the latter bijection becomes an isomorphism of Lie-Rinehart algebras. We have thus proved the main result in this section:
\begin{theorem}
Let $E \to M$ and $V \to M$ be vector bundles. The assignment $D \mapsto D^\ast$ establishes a degree inverting isomorphism of Lie-Rinehart algebras over $\mathfrak D^\bullet_{\mathrm{sym}, \mathrm{core}}(E) \cong C^\infty_{\mathrm{poly}} (E^\ast)$ between linear $V_E$-multivectors and polynomial derivations of $V_{E^\ast}$.
\end{theorem}

\section{Differential Operators and Their Symbols}\label{sec:DO}
We finally come to the object of our primary interest: differential operators. This sections is a super-short review of the subject.
 
Let $V,W \to M$ be vector bundles. A (linear) \emph{differential operator (DO in the following) of order $q$ from $V$ to $W$} is an $\mathbb R$-linear map
\[
\Delta : \Gamma (V) \to \Gamma (W)
\]
such that
\[
[\cdots[[\Delta, f_0], f_1], \cdots, f_q] = 0
\]
for all $f_i \in C^\infty (M)$. In particular, DO of order zero are just vector bundle maps $V \to W$. We denote by $DO_q (V, W)$ the space of order $q$ DOs from $V$ to $W$. Clearly, a DO of order $q$ is also a DO of order $q +1$, and we get the filtration
\[
DO_0 (V, W) = \Gamma (\operatorname{Hom}(V,W)) \subset DO_1 (V, W) \subset \cdots \subset DO_q(V, W) \subset
\]
The union of all $DO_q (V,W)$ will be denoted simply by $DO (V,W)$. A \emph{scalar DO} on $M$ is a DO acting on functions over $M$, i.e.~a DO from the trivial line bundle $\mathbb R_M := M \times \mathbb R$ to itself. We use the symbol $DO_{q}(\mathbb R_E)$ (instead of $DO_{q}(\mathbb R_M, \mathbb R_M)$) for scalar DOs.

The composition of an order $q$ and an order $r$ DO is an order $q + r$ DO. In particular, for all $q$, $DO_q (V,W)$ is a $C^\infty (M)$-module in two different ways: via composition on the left and composition on the right with a function on $M$ (seen as an order $0$ DO). We will consider the first module structure unless otherwise stated. The space $DO(\mathbb R_M)$ is a filtered non-commutative algebra with the composition. It is actually the universal enveloping algebra of the tangent Lie algebroid $TM \to M$. Being an associative algebra, $DO (\mathbb R_M)$ is also a Lie algebra with the commutator. Notice that the commutator of an order $q$ and an order $r$ scalar DO is an order $q + r -1$ scalar DO.

Given an order $q$ DO $\Delta : \Gamma (V) \to \Gamma (W)$ from $V$ to $W$, and functions $f_1, \ldots, f_q$, the nested commutator 
\[
[\cdots[\Delta, f_1], \cdots, f_q]
\]
is an order $0$ DO. Additionally, it is a derivation in each of the arguments $f_i$ and it is symmetric in those argument. In this way, we get a map
\[
\sigma : DO_q (V,W) \mapsto \mathfrak \Gamma \left(S^q TM \otimes \operatorname{Hom}(V,W)\right), \quad \Delta \mapsto \sigma (\Delta)
\]
with
\[
\sigma (\Delta) (f_1, \ldots, f_q) = [\cdots[\Delta, f_1], \cdots, f_q].
\]
The map $\sigma$ is called the \emph{symbol} and it fits in a short exact sequence of $C^\infty (M)$-modules
\begin{equation}\label{eq:ses_symbol}
0 \longrightarrow DO_{q-1}(V, W) \longrightarrow DO_{q}(V, W) \overset{\sigma}{\longrightarrow} \Gamma \left(S^q TM \otimes \operatorname{Hom}(V,W)\right) \longrightarrow 0
\end{equation}
where the second arrow is the inclusion. 

\begin{example}
Vector fields are first order scalar DOs. Derivations of the vector bundle $V$ are first order DOs $D$ from $V$ to itself such that $\sigma (D)$ belongs to $\mathfrak X (M) \subset \Gamma (TM \otimes \operatorname{End} V)$. Additionally we have $\sigma (D) = l_D$.
\end{example}

For scalar DOs the short exact sequence (\ref{eq:ses_symbol}) becomes
\[
0 \longrightarrow DO_{q-1}(\mathbb R_M) \longrightarrow DO_{q}(\mathbb R_M) \overset{\sigma}{\longrightarrow} \mathfrak X^q_{\mathrm{sym}}(M) \longrightarrow 0.
\]
The symbol of scalar DOs intertwines the commutator with the Poisson bracket (of symmetric multivectors) in the sense that
\[
\sigma \big([\Delta', \Delta]\big) = \big\{ \sigma(\Delta), \sigma (\Delta')\big\} 
\]
 whenever $\Delta \in DO_q(\mathbb R_M)$ and $\Delta' \in DO_{q'}(\mathbb R_M)$, in which case we take $[\Delta, \Delta'] \in DO_{q+q'-1}(\mathbb R_M)$. 
 
 We conclude this short review section commenting briefly on the coordinate description of (scalar) DOs. To do this we first fix our conventions on the multi-index notation for multiple partial derivatives. Let $(x^i)$, $i = 1, \ldots, n$ be variables. A length $k$ multi-index $I$ is a word $I = i_1 \ldots i_k$, with $i_j =1, \ldots, n$, where words are considered modulo permutations of their letters. The length $k$ of a multi-index $I = i_1 \cdots i_k$ is also denoted $|I|$. Words can be composed by concatenation and we also consider the empty multi-index $\varnothing$. If we do so, then multi-indexes are elements in the free abelian monoid spanned by $1, \ldots, n$. The lenght is then a monoid homomorphism. A lenght $k$ multi-index $I = i_1 \ldots i_k$, determines an order $k$ DO
 \[
 \frac{\partial^{|I|}}{\partial x^I} := \frac{\partial^k}{\partial x^{i_1} \cdots \partial x^{i_k}}.
 \]
 
 Now, we go back to manifolds $M$ (and vector bundles over them). Actually, $DO_q (\mathbb R_M)$ (likewise $DO_q (V,W)$) is the $C^\infty (M)$-module of sections of a vector bundle over $M$. If $(x^i)$ are coordinates on $M$, then $DO_q (\mathbb R_M)$ is spanned locally (in the corresponding coordinate neighborhood) by 
 \[
 \frac{\partial^{|I|}}{\partial x^I} , \quad |I| = 0, 1, \ldots, q.
 \]
More precisely, locally, every DO $\Delta \in DO_q (\mathbb R_M)$ can be uniquely written in the form
\begin{equation}\label{eq:Delta_loc}
\Delta =\sum_{|I| \leq q} \Delta^I (x) \frac{\partial^{|I|}}{\partial x^I}
\end{equation}
where the $ \Delta^I (x)$ are local functions on $M$. The $\Delta^I (x)$ can be recovered via formulas
\begin{equation}\label{eq:Delta^I}
\Delta^{i_1\cdots i_k} (x) =  \frac{1}{(i_1 \cdots i_k)!} [\cdots[\Delta, x^{i_1}], \cdots, x^{i_k}] (1), \quad k = 0, 1, \ldots, q,
\end{equation}
where, for a multi-index $I$, we denoted by $I!$ the product $I[1]! \cdots I[n]!$ where $I[i]$ is the number of times the letter $i$ occurs in $I$.


Finally, if $\Delta$ is an order $q$ scalar DO locally given by (\ref{eq:Delta_loc}), then its symbol $\sigma (\Delta)$ is locally given by
\[
\sigma (\Delta) = \frac{1}{q!} \Delta^{i_1 \cdots i_q} \frac{\partial}{\partial x^{i_1}} \odot \cdots \odot \frac{\partial}{\partial x^{i_q}}.
\]

\section{Core and Fiber-wise Linear Differential Operators}\label{sec:FWL_DO}
This is the main section of the paper. We propose a notion of \emph{FWL} (scalar) \emph{DO} on the total space of a vector bundle. Our definition is partly motivated by the fact that the symbol of a FWL DO is a FWL multivector. It is also motivated by the \emph{linearization construction} discussed in the next section. Yet another motivating little fact is that the Laplacian of a FWL metric is a FWL DO (Example \ref{ex:g_linear}).

Let $E \to M$ be a vector bundle. We have learnt from Sections \ref{sec:vector_fields}, \ref{sec:multivectors} and \ref{sec:derivations} that, given a type $\mathfrak T$ of geometric structures on manifolds (functions, vector fields, tensors, etc.) appropriate notions of \emph{core} and \emph{FWL} structures of the type $\mathfrak T$ on $E$ exist, and these notions can be identified by means of the following recipe: 1) notice that the space $\mathfrak T(E)$ of structures of type $\mathfrak T$ on $E$ is naturally graded (via the action of multiplicative reals on $E$ by fiber-wise scalar multiplication), 2) identify the smallest degree $k$ for which the degree $k$ homogeneous component $\mathfrak T (E)_k$ of $\mathfrak T (E)$ is non-trivial, and 3) put $\mathfrak T_{\mathrm{core}}(E) = \mathfrak T (E)_k$ and $\mathfrak T_{\mathrm{lin}}(E) = \mathfrak T (E)_{k+1}$. A quick check shows that this recipe cooks up the required definitions in all the cases considered so far. Notice that we could make this recipe much more rigorous adopting for the rather vague ``geometric structure of type $\mathfrak T$'' the very precise notion of \emph{natural vector bundle $\mathfrak T$}, but we will not need this level of abstraction. 

We adopt the strategy described above to define core and FWL DOs on $E$. Consider the non-commutative algebra $DO(\mathbb R_E)$ of scalar DOs $\Delta : C^\infty (E) \to C^\infty (E)$. We begin noticing that, for each $q$, the space $DO_q (\mathbb R_E)$ of DOs $\Delta : C^\infty (E) \to C^\infty (E)$ of order $q$ is naturally graded:
\[
DO_q (\mathbb R_E) = \bigoplus_{k = -q}^\infty DO_q (\mathbb R_E)_k
\]
where $DO_q (\mathbb R_E)_k$ consists of degree $k$ DOs (of order $q$), i.e.~DOs $\Delta$ such that
\[
h_t^\ast (\Delta) = t^{k} \Delta
\]
for all $t \neq 0$. The smallest degree $k$ for which $DO_q (\mathbb R_E)_k$ is non-trivial is $k = -q$. So, following our recipe, we put
\[
DO_{q, \mathrm{core}} (E) := DO_q (\mathbb R_E)_{-q},
\]
and call them \emph{core DOs}. We also put
\begin{equation}\label{eq:loc_DO_core}
DO_{\mathrm{core}}(E) := \bigoplus_{q} DO_{q, \mathrm{core}} (E).
\end{equation}
Let $(x^i, u^\alpha)$ be vector bundle coordinates on $E$, and let $(x^i, u_\alpha)$ be dual coordinates on $E^\ast$. A DO $F \in DO_q (\mathbb R_E)$ is a core DO if and only if, locally,
\begin{equation}\label{eq:core_DO}
F = \sum_{|A| = q}F^A (x) \frac{\partial^{|A|}}{\partial u^A},
\end{equation}
where $A = \alpha_1 \cdots \alpha_q$ is a lenght $q$ multi-index. 

It follows from (\ref{eq:core_DO}) that $DO_{\mathrm{core}}(E) \subset DO (\mathbb R_E)$ is the subalgebra spanned by core functions $C^\infty_{\mathrm{core}}(M)$ and core vector fields $\mathfrak X_{\mathrm{core}}(E)$. Equivalently, it is the universal enveloping algebra of the abelian Lie algebroid $E \Rightarrow M$. Because of the latter description, there is an algebra isomorphism
\[
\Gamma (S^\bullet E) \to DO_{\mathrm{core}}(E),
\]
mapping a monomial
\[
e_1 \odot \cdots \odot e_q, \quad e_i \in \Gamma (E),
\]
to the DO
\[
e_1^\uparrow \circ \cdots \circ e_q^\uparrow.
\]
In its turn, as already mentioned, $\Gamma (S^\bullet E)$ identifies with polynomial functions on $E^\ast$. 
In the following we will often identify $DO_{\mathrm{core}}(E)$ with both $\Gamma (S^\bullet E)$ and $C^\infty_{\mathrm{poly}} (E^\ast)$ via the latter isomorphisms. If $F \in DO_{q, \mathrm{core}}(E)$ is locally given by (\ref{eq:loc_DO_core}), then it identifies with 
\[
\frac{1}{q!}F^{\alpha_1 \cdots \alpha_q} (x) u_{\alpha_1} \odot \cdots \odot u_{\alpha_q} \in \Gamma (S^q E),
\]
and 
\[
\sum_{|A| = q}F^A (x) u_A = \frac{1}{q!}F^{\alpha_1 \cdots \alpha_q} (x) u_{\alpha_1} \cdots u_{\alpha_q}\in C^\infty (E^\ast)_q,
\]
where, for $A = \alpha_1 \cdots \alpha_q$, we denoted by $u_A$ the monomial $u_{\alpha_1} \cdots u_{\alpha_q}$.

We will always consider $DO (\mathbb R_E)$ as a $DO_{\mathrm{core}}(E)$-module with the scalar multiplication given by the \emph{left} composition. 

We now pass to \emph{FWL DOs}. Following our recipe again, for each $q$ we put
\[
DO_{q, \mathrm{lin}}(E) := DO_q(\mathbb R_E)_{-q+1}.
\]

\begin{definition}\label{def:FWL_DO}
DOs in $DO_{q, \mathrm{lin}}(E)$ are called  \emph{fiber-wise linear differential operators} (FWL DOs) of order $q$.
\end{definition}

For instance, $DO_{0, \mathrm{lin}}(E) = C^\infty_{\mathrm{lin}}(E)$, and $DO_{1, \mathrm{lin}}(E) = \mathfrak X_{\mathrm{lin}}(E) \oplus C^\infty (M)$. It is also clear that $DO_{q, \mathrm{lin}}(E) \supset DO_{q-1, \mathrm{core}}$ for all $q$. More precisely
\[
DO_{q-1, \mathrm{core}} = DO_{q, \mathrm{lin}}(E) \cap DO_{q-1}(\mathbb R_E).
\]

We put
\[
DO_{\mathrm{lin}}(E) := \bigoplus_{q} DO_q(\mathbb R_E)_{-q+1}.
\]
Clearly $DO_{\mathrm{core}}(E) \subseteq DO_{\mathrm{lin}}(E) \subseteq \subseteq DO (\mathbb R_E)$.

A DO $\Delta \in DO_{q}(\mathbb R_E)$ is FWL if and only if, in vector bundle coordinates, it looks like
\begin{equation}\label{eq:loc_DO_stab}
\Delta = \sum_{|A| = q -1} \Delta^{i|A}(x) \frac{\partial^{|A| + 1}}{\partial x^i \partial u^A} + \sum_{|B| = q} \Delta^B_\alpha (x) u^\alpha \frac{\partial^{|B|}}{\partial u^B}
 + \sum_{|C| = q-1} \Delta^C (x) \frac{\partial^{|C|}}{\partial u^C}.
\end{equation}
It is easy to see from this formula that $DO_{\mathrm{lin}}(E) \subset DO (\mathbb R_E)$ is the $DO_{\mathrm{core}}(E)$-submodule spanned by $1$, $C^\infty_{\mathrm{lin}} (E)$ and $\mathfrak X_{\mathrm{lin}}(E)$.

\begin{example}\label{ex:g_linear}
Let $g$ be a metric on $E$, and assume it is FWL. Then, the associated Laplacian operator $\Delta_g : C^\infty (E) \to C^\infty (E)$ is a FWL DO operator (of order $2$). One can see this working in vector bundle coordinates. But there is also a (basically) coordinate free proof that we now illustrate. First of all, from $g$ being FWL, it immediately follows that the inverse tensor $g^{-1}$ is FWL as well. Now, the covariant derivative $\nabla \theta $ of a $1$-form $\theta$ along the Levi-Civita connection $\nabla$ is the covariant $2$-tensor given by the formula:
\begin{equation}\label{eq:cov_der}
\nabla \theta = \frac{1}{2} \left(d \theta + \mathcal L_{\sharp (\theta)} g \right), 
\end{equation}
where $\sharp : T^\ast E \to TE$ is the musical isomorphism. Equivalently, the covariant derivative $\nabla_X Y$ of a vector field $Y$ along another vector field $X$ is the vector field $\nabla_X Y$ that acts on functions $f \in C^\infty (M)$ as follows:
\begin{equation}\label{eq:cov_der_2}
\nabla_X Y (f) = X(Y(f)) - \frac{1}{2}(\mathcal L_{\operatorname{grad} f} g)(X,Y)
\end{equation}
where $\operatorname{grad} f = \sharp (df)$ is the \emph{gradient} of $f$. Using (\ref{eq:cov_der}) (or (\ref{eq:cov_der_2})) and the naturality of both the de Rham differential and the Lie derivative, it is easy to see that the covariant derivative of arbitrary tensor fields commutes with the pull back along $h_t$ for all $t \neq 0$. As the Laplacian $\Delta_g f$ of a function $f$ is obtained by contracting the covariant derivative of $df$ with $g^{-1}$, then $\Delta_g$ decreases by one the degree of a homogeneous (fiber-wise polynomial) function. So it is a second order $DO$ of degree $1 - 2 = -1$, i.e.~a FWL DO of order 2, as claimed. It might be also interesting to remark that the Levi-Civita connection of a FWL metric is a FWL connection according to a definition introduced in \cite{PSV2020}.
\end{example}

\begin{lemma}\label{lem:stabilizer}
The space $DO_{\mathrm{lin}}(E)$ of linear DOs is the stabilizer Lie subalgebra of $DO_{\mathrm{core}}(E)$, i.e.~a DO $\Delta \in DO (\mathbb R_E)$ is in $DO_{\mathrm{lin}}(E)$ if and only if $[\Delta, F] \in DO_{\mathrm{core}}(E)$ for all $F \in DO_{\mathrm{core}}(E)$.
\end{lemma}

\begin{proof}
The ``only if part'' of the statement immediately follows from an obvious order/degree argument. For the ``if part'', consider a DO $\Delta$ of order $r$. Locally,
\[
\Delta = \sum_{|I| + |A| \leq r} \Delta^{I|A} (x, u)\frac{\partial^{|I| + |A|}}{\partial x^I \partial u^A}.
\]
Assume that $\Delta$ is in the stabilizer of $DO_{\mathrm{core}}(E)$. We want to show that $\Delta$ is the sum of operators of the form (\ref{eq:loc_DO_stab}) (with possibly varying $q \leq r$). As $x^i$ is a core function for all $i$, the commutator $[\Delta, x^i]$ is a core DO. But
\[
[\Delta, x^i] = \sum_{|J| + |A| \leq r - 1} \left(J[i] +1 \right)\Delta^{Ji|A} (x, u)\frac{\partial^{|J| + |A| - 1}}{\partial x^J \partial u^A},
\]
so it can only be a core DO for all $i$ if 1) $\Delta^{I|A}(x,u) = 0$ for $|I| > 1$, and 2) $\Delta^{i|A}(x,u) = \Delta^{i|A}(x)$. In other words, $\Delta$ is necessarily of the form
\begin{equation}\label{eq:partial_form}
\Delta = \sum_{|A| \leq r -1} \Delta^{i|A}(x) \frac{\partial^{|A| + 1}}{\partial x^i \partial u^A} + \sum_{|B| \leq k} \Delta^B(x,u) \frac{\partial^{|B|}}{\partial u^B}.
\end{equation}
and, to conclude, it is enough to prove that $\Delta^B(x,u)$ is of the form $\Delta^B(x,u) = \Delta^B_\alpha (x) u^\alpha + \Delta^B (x)$. To do this, recall that $\frac{\partial}{\partial u^\alpha}$ is a core vector field for all $\alpha$, hence $\left[\Delta, \frac{\partial}{\partial u^\alpha}\right]$ is a core DO. But, from (\ref{eq:partial_form}), 
\[
\left[\Delta, \frac{\partial}{\partial u^\alpha}\right] = - \frac{\partial \Delta^B (x,u)}{\partial u^\alpha} \frac{\partial^{|B|}}{\partial u^B},
\]
which is a core DO for all $\alpha$ if and only if 
\[
\frac{\partial^2 \Delta^B (x,u)}{\partial u^\alpha \partial u^\beta} = 0
\]
for all $\alpha, \beta$, i.e.~$\Delta^B(x,u)$ is a (non-necessarily homogeneous) first order polynomial in the variables $u$, as desired.
\end{proof}

It follows from Lemma \ref{lem:stabilizer} that $DO_{\mathrm{lin}}(E)$ is a Lie subalgebra in $DO (\mathbb R_E)$. As already mentioned, it is also a $DO_{\mathrm{core}}$-submodule. Actually, it is a Lie-Rinehart algebra over $DO_{\mathrm{core}}(E)$, the anchor being the adjoint operator $ad : \Delta \mapsto ad(\Delta):= [\Delta, -]$. To see this, first notice that $ad(\Delta)$ is indeed a well-defined derivation of $DO_{\mathrm{core}}(E)$ for all $\Delta \in DO_{\mathrm{lin}}(E)$. Now, take $\Delta, \Delta' \in DO_{\mathrm{lin}}(E)$ and $F, F' \in DO_{\mathrm{core}} (E)$, and compute
\[
\begin{aligned}
{}[\Delta, F \circ \Delta'] & = ad(\Delta)(F) \circ \Delta' + F \circ [\Delta, \Delta'], \\
ad(F \circ \Delta)(F') & = [F \circ \Delta, F'] = F \circ [\Delta, F'] = F \circ ad(\Delta)(F').
\end{aligned}
\]
Finally, from $DO_{\mathrm{core}}(E) \cong C^\infty_{\mathrm{poly}}(E^\ast)$, we see that, for every $\Delta \in DO_{\mathrm{lin}}(E)$, the derivation $ad(\Delta)$ determines a polynomial vector field (of the same degree) on $E^\ast$, also denoted $ad(\Delta)$.

\begin{theorem}\label{theor:stabilizer}
The sequence of Lie-Rinehart algebras
\begin{equation}\label{eq:ex_seq}
0 \longrightarrow DO_{\mathrm{core}}(E) \longrightarrow DO_{\mathrm{lin}}(E) \overset{ad}{\longrightarrow} \mathfrak X_{\mathrm{poly}} (E^\ast) \longrightarrow 0
\end{equation}
is exact. 
\end{theorem}

\begin{proof}
First of all, as already remarked, $DO_{\mathrm{core}}(E)$ is in $DO_{\mathrm{lin}}(E)$. Even more, as it is an abelian subalgebra in $DO (\mathbb R_E)$, then it is actually in the kernel of $ad : DO_{\mathrm{lin}}(E) \to \mathfrak X_{\mathrm{poly}}(E^\ast)$. To see that core DOs exhaust the kernel of $ad$ (i.e.~$DO_{\mathrm{core}}(E)$ is its own centralizer), assume that $[\Delta, F] = 0$ for all $F \in DO_{\mathrm{core}}(E)$. Then, exactly the same computation as in the proof of Lemma \ref{lem:stabilizer} shows that $\Delta$ is locally of the form (\ref{eq:loc_DO_stab}) with $\Delta^{i|A} (x) = \Delta^B_{\alpha}(x) = 0$, i.e.~$\Delta \in DO_{\mathrm{core}}(E)$. For the exactness of the sequence (\ref{eq:ex_seq}) it remains to show that the map $ad : DO_{\mathrm{lin}}(E) \to \mathfrak X_{\mathrm{poly}}(E^\ast)$ is surjective. To do that, we work in local coordinates again. So, let $(x^i, u^\alpha)$ be vector bundle coordinates on $E$, and let $(x^i, u_\alpha)$ be dual coordinates on $E^\ast$. It is not hard to see that, if $\Delta$ is locally given by (\ref{eq:loc_DO_stab}), 
then the vector field $ad(\Delta)$ is locally given by
\begin{equation}\label{eq:ad_Delta}
ad(\Delta) =  \sum_{|A| = q -1} \Delta^{i|A}(x) u_A\frac{\partial}{\partial x^i} - \sum_{|B| = q} \Delta^B_\alpha (x)u_B \frac{\partial}{\partial u_\alpha},
\end{equation}
where, for a multi-index $A = \alpha_1 \cdots \alpha_{s}$, we denoted by $u_A$ the monomial $u_{\alpha_1} \cdots u_{\alpha_s}$ ($s =q, q-1$). As (\ref{eq:ad_Delta}) is the local expression of a generic homogeneous polynomial vector field of degree $q-1$, we are done.
\end{proof}

Our next aim is proving that the Lie-Rinehart algebra $DO_{\mathrm{lin}}(E)$ is canonically isomorphic to the Lie-Rinehart algebra of polynomial derivations of an appropriate line bundle on $E^\ast$. We begin with a simple

\begin{proposition}\label{prop:symbol}
The symbol $\sigma (\Delta)$ of a FWL DO $\Delta \in DO_{q, \mathrm{lin}}(E)$ is a FWL symmetric $q$-multivector field. Every FWL symmetric $q$-multivector field is the symbol of an order $q$ FWL DO.
\end{proposition}

\begin{proof}
The statement immediately follows from (\ref{eq:loc_DO_stab}) and the easy fact that a symmetric $q$-multivector $P$ is FWL if and only if, in vector bundle coordinates, it is of the form
\[
P = P^{i|\alpha_1 \cdots \alpha_{q-1}}(x) \frac{\partial}{\partial x^i} \odot \frac{\partial}{\partial u^{\alpha_1}} \odot \cdots \odot \frac{\partial}{\partial u^{\alpha_{q-1}}} + P^{\beta_1 \cdots \beta_q}_\alpha (x) u^\alpha \frac{\partial}{\partial u^{\beta_1}} \odot \cdots \odot \frac{\partial}{\partial u^{\beta_{q}}}.
\]
\end{proof}

Now let $\Delta \in DO_{q, \mathrm{lin}}(E)$. Notice that the adjoint operator $ad (\Delta)$, seen as a polynomial vector field on $E^\ast$, corresponds exactly to the symbol $\sigma (\Delta)$ via the isomorphism $\mathfrak X^q_{\mathrm{sym}, \mathrm{lin}}(E) \cong \mathfrak X (E^\ast)_{q-1}$. It is also clear that, in view of its coordinate form (\ref{eq:loc_DO_stab}), $\Delta$ is completely determined by $\sigma (\Delta)$ or, equivalently, $ad(\Delta)$, together with the map
\[
\Psi_\Delta : \underset{\text{$q-1$ times}}{\underbrace{\Gamma (E^\ast) \times \cdots \times \Gamma (E^\ast)}} \to C^\infty (M), \quad (\varphi_1, \ldots, \varphi_{q-1}) \mapsto [ \cdots [\Delta, \ell_{\varphi_1}], \cdots, \ell_{\varphi_{q-1}}](1).
\]
The map $\Psi_\Delta$ is clearly well-defined. Additionally, it enjoys the following properties
\begin{enumerate}
\item $\Psi_\Delta$ is symmetric,
\item $\Psi_\Delta$ is a first order DO in each entry. 
\end{enumerate}
More precisely, we have the following
\begin{lemma}
The map $\Psi_\Delta$ satisfies
\[
\Psi_\Delta (\varphi_1, \ldots, \varphi_{q-2}, f \varphi_{q-1}) = f \Psi_\Delta (\varphi_1, \ldots, \varphi_{q-2},  \varphi_{q-1}) + l_{\sigma (\Delta)}(\varphi_1, \ldots, \varphi_{q-1})(f),
\]
for all $\varphi_i \in \Gamma (E^\ast)$ and $f \in C^\infty (M)$.
\end{lemma}

\begin{proof}
Let $\varphi_i$ and $f$ be as in the statement, and compute
\[
\begin{aligned}
& \Psi_\Delta (\varphi_1, \ldots, \varphi_{q-2}, f \varphi_{q-1}) \\
& = [[\cdots [ \Delta, \ell_{\varphi_1}], \cdots, \ell_{\varphi_{q-2}}], f \ell_{\varphi_{q-1}}](1) \\
& = f[[\cdots [ \Delta, \ell_{\varphi_1}], \cdots, \ell_{\varphi_{q-2}}], \ell_{\varphi_{q-1}}](1) + [\cdots [ \Delta, \ell_{\varphi_1}], \cdots, \ell_{\varphi_{q-2}}] (\ell_{\varphi_{q-1}}) \\
& = f\Psi_\Delta (\varphi_1, \ldots, \varphi_{q-2}, f \varphi_{q-1}) +  [[\cdots [ \Delta, \ell_{\varphi_1}], \cdots, \ell_{\varphi_{q-2}}], f] (\ell_{\varphi_{q-1}})
\end{aligned}
\]
It remains to compute the last summand. So
\[
\begin{aligned}
& [[\cdots [ \Delta, \ell_{\varphi_1}], \cdots, \ell_{\varphi_{q-2}}], f] (\ell_{\varphi_{q-1}}) \\
& = [[[\cdots [ \Delta, \ell_{\varphi_1}], \cdots, \ell_{\varphi_{q-2}}], f],\ell_{\varphi_{q-1}}] + \ell_{\varphi_{q-1}} [[\cdots [ \Delta, \ell_{\varphi_1}], \cdots, \ell_{\varphi_{q-2}}], f] (1) \\
& = l_{\sigma (\Delta)}(\varphi_1, \ldots, \varphi_{q-1})(f),
\end{aligned}
\]
where we used that, from (\ref{eq:loc_DO_stab}) again, the last summand in the second line is necessarily zero. This concludes the proof.
\end{proof}

The data $(\sigma(\Delta), \Psi_\Delta)$ (determine $\Delta$ completely and) can be repackaged in a very useful way. Namely, consider the line bundle
\[
L = \wedge^{\mathrm{top}} E.
\]
Then the pair $(\sigma(\Delta), \Psi_\Delta)$ determines a FWL $q$-$L_E$-multivector in the following way. Recall that a FWL $q$-$L_E$-multivector can be equivalently presented as a pair $(P, \Phi)$ consisting of a FWL symmetric multivector $P \in \mathfrak X^\bullet_{\mathrm{sym}, \mathrm{lin}} (E)$ and a vector bundle map $\Phi : S^\bullet E^\ast \to DL$, such that $l_P = l \circ \Phi$. We claim that we can construct such a pair from the pair $(\sigma(\Delta), \Psi_\Delta)$. Namely, we put
\[
P = \sigma(\Delta)
\]
 and define $\Phi = \Phi_\Delta$ by putting
 \begin{equation}\label{eq:Phi_Delta}
\Phi_\Delta (\varphi_1, \ldots, \varphi_{q-1} ) (U) := \sigma (\Delta) (\varphi_1, \ldots, \varphi_{q-1}, -) U + \Psi_\Delta (\varphi_1, \ldots, \varphi_{q-1}) U,
 \end{equation}
for all $\varphi_i \in \Gamma (E^\ast)$ and $U \in \Gamma (L)$. Equation (\ref{eq:Phi_Delta}) needs some explanations. In the first summand of the rhs, we interpret $\sigma (\Delta)$ as a $q$-multiderivation of $E^\ast$, so, when contracting it with the $q-1$ sections $\varphi_1, \ldots, \varphi_{q-1}$, we get a plain derivation $\sigma (\Delta) (\varphi_1, \ldots, \varphi_{q-1}, -) \in \mathfrak D (E^\ast) \cong \mathfrak D (E)$. As already remarked, derivations of $E$ act on the exterior algebra of $E$. In our case we have
 \[
 D (e_1 \wedge \cdots \wedge e_{\mathrm{top}}) = \sum_{i} e_1 \wedge \cdots \wedge De_i \wedge \cdots \wedge e_{\mathrm{top}}.
 \]
 for all $D \in \mathfrak D (E)$, and all $e_i \in \Gamma (E)$.
 
 The next theorem is the main result of the paper.
 
 \begin{theorem}\label{theor:iso_DO_D_sym}
The assignment $\Delta \mapsto (\sigma(\Delta), \Phi_\Delta)$ establishes a degree inverting isomorphism of Lie-Rinehart algebras $A: DO_{\mathrm{lin}}(E) \to \mathfrak D_{\mathrm{sym}, \mathrm{lin}}(L_E)$. 
\end{theorem}

\begin{proof}
First of all, we have to show that $\Phi_\Delta$ is well-defined, i.e.~it is symmetric and $C^\infty (M)$-linear in all its arguments. The symmetry is obvious. For the linearity, let $f \in C^\infty (M)$, and compute
\[
\begin{aligned}
& \Phi_\Delta (\varphi_1, \ldots, \varphi_{q-2}, f \varphi_{q-1}) (U) \\
& = {\sigma (\Delta) (\varphi_1, \ldots, \varphi_{q-2}, f \varphi_{q-1}, -)} U + \Psi_\Delta (\varphi_1, \ldots, \varphi_{q-2}, f\varphi_{q-1}) U.
\end{aligned}
\] 
Let us compute the two summands separately. First of all, for every $\varphi \in \Gamma (E^\ast)$,
\[
\sigma (\Delta) (\varphi_1, \ldots, \varphi_{q-2}, f \varphi_{q-1}, \varphi) = l_{\sigma(\Delta)}(\varphi_1, \ldots, \varphi_{q-2}, \varphi)(f)\varphi_{q-1} + f \sigma(\Delta)(\varphi_1, \ldots, \varphi_{q-1}, \varphi)
\]
showing that
\[
\sigma (\Delta) (\varphi_1, \ldots, \varphi_{q-2}, f \varphi_{q-1}, -) = \varphi_{q-1} \otimes e + f \sigma (\Delta)(\varphi_1, \ldots, \varphi_{q-1}, -),
\]
where $e \in \Gamma (E)$ is the section implicitly defined by
\[
\langle \varphi, e \rangle = l_{\sigma(\Delta)}(\varphi_1, \ldots, \varphi_{q-2}, \varphi)(f)
\]
for all $\varphi \in \Gamma (E^\ast)$. We remark for future use that, in particular, 
\begin{equation}\label{eq:varphi_q-1,e}
\langle \varphi_{q-1}, e \rangle = l_{\sigma(\Delta)}(\varphi_1, \ldots, \varphi_{q-1})(f).
\end{equation}
Now
\begin{equation}\label{eq:3}
{\sigma (\Delta) (\varphi_1, \ldots, \varphi_{q-2}, f \varphi_{q-1}, -)} U 
 = (\varphi_{q-1} \otimes e)U + f \sigma (\Delta)(\varphi_1, \ldots, \varphi_{q-1}, -) U. 
\end{equation}
The endomorphism $\varphi_{q-1} \otimes e : E^\ast \to E^\ast$ act on $E$ via its dual, which is minus its transpose, hence
\begin{equation}\label{eq:1}
(\varphi_{q-1} \otimes e) U = - \operatorname{trace}(e \otimes \varphi_{q-1}) U = - \langle \varphi_{q-1}, e \rangle U = -l_{\sigma(\Delta)}(\varphi_1, \ldots, \varphi_{q-1})(f) U.
\end{equation}
where we used (\ref{eq:varphi_q-1,e}). 
Substituting (\ref{eq:1}) into (\ref{eq:3}), we find
\[
{\sigma (\Delta) (\varphi_1, \ldots, \varphi_{q-2}, f \varphi_{q-1}, -)} U = f{\sigma (\Delta)(\varphi_1, \ldots, \varphi_{q-1}, -)} U -l_{\sigma(\Delta)}(\varphi_1, \ldots, \varphi_{q-1})(f) U.
\]
We also have
\[
\Psi_\Delta (\varphi_1, \ldots, \varphi_{q-2}, f\varphi_{q-1}) U = f \Psi_\Delta (\varphi_1, \ldots, \varphi_{q-1}) U + l_{\sigma(\Delta)}(\varphi_1, \ldots, \varphi_{q-1})(f) U,
\]
and putting everything together we find
\[
\begin{aligned}
\Phi_\Delta (\varphi_1, \ldots, \varphi_{q-2}, f \varphi_{q-1}) U & = f {\sigma (\Delta)(\varphi_1, \ldots, \varphi_{q-1}, -)} U + f \Psi_\Delta (\varphi_1, \ldots, \varphi_{q-1}) U \\
& = f \Phi_\Delta (\varphi_1, \ldots, \varphi_{q-1}) U.
\end{aligned}
\]
We conclude that $\Phi_\Delta$ is a vector bundle map $S^{q-1} E^\ast \to DL$ as desired. Additionally, the composition $l \circ \Phi_\Delta$ does clearly agree with $\sigma (\Delta)$ so that $(\sigma (\Delta), \Phi_\Delta)$ is indeed a $q$-$L_E$-multivector. It is also clear that $\Psi_\Delta$ can be reconstructed from $(\sigma (\Delta), \Phi_\Delta)$ showing that the correspondence $\Delta \mapsto (\sigma (\Delta), \Phi_\Delta)$ is injective. Next we prove the $DO_{\mathrm{core}}(E)$-linearity. So, take $F \in DO_{p, \mathrm{core}}(E) = \Gamma (S^p E) = C^\infty (E^\ast)_p$, and $\Delta \in DO_{q,\mathrm{lin}}(E)$, so that $F \circ \Delta \in DO_{p + q,\mathrm{lin}}(E)$. We want to show that $(\sigma(F \circ \Delta), \Phi_{F \circ \Delta}) = F \cdot (\sigma (\Delta), \Phi_\Delta)$. To do this we begin noticing that the product $F \cdot (\sigma (\Delta), \Phi_\Delta)$ is the pair $(D', \Phi')$ where $D' \in \mathfrak X^{p+q}_{\mathrm{sym}, \mathrm{lin}}(E)$ is the symmetric multivector that, when interpreted as a multiderivation $D' : \Gamma (E^\ast) \times \cdots \times \Gamma (E^\ast) \to \Gamma (E^\ast)$, is given by
\[
\begin{aligned}
D' (\varphi_1, \ldots, \varphi_{p+q}) & = F\cdot \sigma(\Delta)(\varphi_1, \ldots, \varphi_{p+q}) \\
& = \sum_{\sigma \in S_{p,q}} \left\langle F, \varphi_{\sigma(1)} \odot \cdots \odot \varphi_{\sigma(p)} \right\rangle \sigma(\Delta) (\varphi_{\sigma (p+1)}, \ldots, \varphi_{\sigma(p+q)}),
\end{aligned}
\]
and, similarly, $\Phi ' : S^{p+q-1} E^\ast \to DL$ is the bundle map given by
\[
\begin{aligned}
\Phi' (\varphi_1, \ldots, \varphi_{p+q-1}) & = F \cdot \Phi_\Delta (\varphi_1, \ldots, \varphi_{p+q-1})\\
&  = \sum_{\sigma \in S_{p,q-1}} \left\langle F, \varphi_{\sigma(1)} \odot \cdots \odot \varphi_{\sigma(p)} \right\rangle \Phi_\Delta (\varphi_{\sigma (p+1)}, \ldots, \varphi_{\sigma(p+q-1)})
\end{aligned}
\]
for all $\varphi_i \in \Gamma (E^\ast)$. From the properties of the symbol map, we have $\sigma (F \circ \Delta) = F \cdot \sigma (\Delta)$ and it remains to take care of $\Phi_{F \circ \Delta}$. So choose $\varphi_i \in \Gamma (E^\ast)$, and compute $\Phi_{F \circ \Delta}(\varphi_1, \ldots, \varphi_{p+q-1})$. From symmetry, it is enough to choose $\varphi_i = \varphi$ for all $i$ and some $\varphi$. First of all, we have
\[
\begin{aligned}
\Psi_{F \circ \Delta}(\, \underset{\text{\tiny{$p+q-1$ times}}}{\underbrace{\varphi\, ,\,  \ldots\, ,\,  \varphi}}\, ) & = [ \cdots [ F \circ \Delta, \underset{\text{\tiny{$p+q-1$ times}}}{\underbrace{\varphi], \cdots, \varphi]}}(1) \\
& = \sum_{l+m = p+q-1} \frac{1}{l! m!} [ \cdots [ F, \underset{\text{\tiny{$l$ times}}}{\underbrace{\varphi], \cdots, \varphi]}} \circ [ \cdots [ \Delta, \underset{\text{\tiny{$m$ times}}}{\underbrace{\varphi], \cdots, \varphi]}} (1)
\end{aligned}.
\]
Only the terms with $l = p, p-1$ (hence $m = q-1, q$, respectively) survive, and we get
\begin{align}
\Psi_{F \circ \Delta}(\, \underset{\text{\tiny{$p+q-1$ times}}}{\underbrace{\varphi\, ,\,  \ldots\, ,\,  \varphi}}\, ) &
 = \frac{1}{p! (q-1)!} [ \cdots [ F, \underset{\text{\tiny{$p$ times}}}{\underbrace{\varphi], \cdots, \varphi]}} \circ [ \cdots [ \Delta, \underset{\text{\tiny{$q-1$ times}}}{\underbrace{\varphi], \cdots, \varphi]}} (1) \nonumber \\
& \quad + \frac{1}{(p-1)! q!} [ \cdots [ F, \underset{\text{\tiny{$p-1$ times}}}{\underbrace{\varphi], \cdots, \varphi]}} \circ [ \cdots [ \Delta, \underset{\text{\tiny{$q$ times}}}{\underbrace{\varphi], \cdots, \varphi]}} (1) \nonumber \\
& = \frac{1}{p! (q-1)!} \Big\langle F, \underset{\text{\tiny{$p$ times}}}{\underbrace{\varphi \odot \cdots \odot \varphi}} \Big\rangle \Psi_\Delta \big(\underset{\text{\tiny{$q-1$ times}}}{\underbrace{\varphi, \ldots, \varphi}}\big) \nonumber \\
& \quad + \frac{1}{(p-1)! q!} \Big\langle F, \underset{\text{\tiny{$p-1$ times}}}{\underbrace{\varphi \odot \cdots \odot \varphi}} \odot \sigma(\Delta) \big(\underset{\text{\tiny{$q$ times}}}{\underbrace{\varphi, \ldots, \varphi}}\big)\Big\rangle. \label{eq:Psi_FDelta}
\end{align}
Now, for all $U \in \Gamma (L)$
\[
\Phi_{F \circ \Delta}(\, \underset{\text{\tiny{$p+q-1$ times}}}{\underbrace{\varphi\, ,\,  \ldots\, ,\,  \varphi}}\, ) U
 = {\sigma(F \circ \Delta) (\, \underset{\text{\tiny{$p+q-1$ times}}}{\underbrace{\varphi\, ,\,  \ldots\, ,\,  \varphi}}\, , -)}U + \Psi_{F \circ \Delta}(\, \underset{\text{\tiny{$p+q-1$ times}}}{\underbrace{\varphi\, ,\,  \ldots\, ,\,  \varphi}}\, )U .
\]
We already computed the second summand, while the first summand is
\begin{align}
{\sigma(F \circ \Delta) (\, \underset{\text{\tiny{$p+q-1$ times}}}{\underbrace{\varphi\, ,\,  \ldots\, ,\,  \varphi}}\, , -)}U & = \frac{1}{p! (q-1)!}{\big\langle F, \underset{\text{\tiny{$p$ times}}}{\underbrace{\varphi \odot \cdots \odot \varphi}} \big\rangle \sigma(\Delta) (\, \underset{\text{\tiny{$q-1$ times}}}{\underbrace{\varphi\, ,\,  \ldots\, ,\,  \varphi}}\, , -)} U\nonumber \\
& \quad + \frac{1}{(p-1)! q!}{\big\langle F, \underset{\text{\tiny{$p-1$ times}}}{\underbrace{\varphi \odot \cdots \odot \varphi}}\, \odot\, - \big\rangle \sigma(\Delta) (\, \underset{\text{\tiny{$q$ times}}}{\underbrace{\varphi ,  \ldots,  \varphi}})}U \nonumber\\
& = \frac{1}{p! (q-1)!}\big\langle F, \underset{\text{\tiny{$p$ times}}}{\underbrace{\varphi \odot \cdots \odot \varphi}} \big\rangle {\sigma(\Delta) (\, \underset{\text{\tiny{$q-1$ times}}}{\underbrace{\varphi\, ,\,  \ldots\, ,\,  \varphi}}\, , -)}U \nonumber \\
& \quad - \frac{1}{(p-1)! q!}\big\langle F, \underset{\text{\tiny{$p-1$ times}}}{\underbrace{\varphi \odot \cdots \odot \varphi}} \odot \sigma(\Delta) (\, \underset{\text{\tiny{$q$ times}}}{\underbrace{\varphi ,  \ldots,  \varphi}}) \big\rangle U.\label{eq:L_FDelta}
\end{align}
From (\ref{eq:Psi_FDelta}) and (\ref{eq:L_FDelta}) it easily follows that $\Phi_{F \circ \Delta} = F \cdot \Phi_\Delta$ as claimed.

The surjectivity of the map $\Delta \mapsto (\sigma(\Delta), \Phi_\Delta)$ now follows from (local) dimension counting.

It remains to check that the isomorhism $A: DO_{\mathrm{lin}}(E) \to \mathfrak D_{\mathrm{sym}, \mathrm{lin}}(L_E)$ defined in this way is both anchor and bracket preserving. For the anchor, the anchor of $(\sigma (\Delta), \Phi_\Delta)$ is the derivation of $DO_{\mathrm{core}}(E) = \Gamma (S^\bullet E) = C^\infty_{\mathrm{poly}}(E)$ corresponding to the linear multivector $\sigma (\Delta)$, which is exactly $ad(\Delta)$.

 For the bracket, as we already discussed $C^\infty_{\mathrm{poly}}(E)$-linearity and compatibility with the anchor, it is enough to discuss the brackets of generators. As already remarked, $DO_{\mathrm{lin}}(E)$ is generated (over $DO_{\mathrm{core}}(E)$) by $1$, $C^\infty_{\mathrm{lin}}(E)$ and $\mathfrak X_{\mathrm{lin}}(E)$. A direct check shows that
 \begin{equation}\label{eq:A_generators}
 A(1) = (0, 1), \quad A(\ell_{\varphi}) = (- \varphi^\uparrow, 0), \quad A(X) = (X^\ast, D_X)
 \end{equation}
 for all $\varphi \in \Gamma (E^\ast)$ and all $X \in \mathfrak X_{\mathrm{lin}}(E)$. Here $X^\ast$ and $D_X$ are, respectively, the linear vector field on $E^\ast$, and the derivation of $L = \wedge^{\mathrm{top}}E$ corresponding to $X$. It is now easy to check that the brackets are preserved on these generators, and this concludes the proof.
\end{proof}

Composing with the isomorphism $\mathfrak D^\bullet_{\mathrm{sym}, \mathrm{lin}}(L_E) \cong \mathfrak D (L_{E^\ast})$ we get a (degree inverting) Lie-Rinehart algebra isomorphism 
\[
DO_{\mathrm{lin}} (E) \overset{\cong}{\longrightarrow} \mathfrak D (L_{E^\ast})
\]
that we denote by $A$ again.

\begin{remark}
Let $(x^i, u^\alpha)$ be vector bundle coordinates on $E$, and let $(x^i, u_\alpha)$ be dual coordinates on $E^\ast$. Denote by $\mathrm{Vol}_u = u_1 \wedge \cdots \wedge u_{\mathrm{top}}$ the local coordinate generator of $\Gamma(L)$. It is easy to check using, e.g., (\ref{eq:ad_Delta}), (\ref{eq:A_generators}), and the $C^\infty_{\mathrm{poly}}(E^\ast)$-linearity, that, if the operator $\Delta \in DO_{\mathrm{lin}}(E)$ is locally given by (\ref{eq:loc_DO_stab}), then the corresponding derivation $A(\Delta) \in \mathfrak D (L_{E^\ast})$ maps a local section $\lambda = f(x,u) \mathrm{Vol}_u$ of $\Gamma (L_{E^\ast})$ to
\[
\begin{aligned}
& A(\Delta)(\lambda) \\
& =  \left( \sum_{|A| = q -1} \Delta^{i|A}(x) u_A\frac{\partial f}{\partial x^i}(x,u) - \sum_{|B| = q} \Delta^B_\alpha (x)u_B \frac{\partial f}{\partial u_\alpha} (x,u) + \sum_{|C| = q-1 }\Delta^C(x)u_C f(x,u) \right) \mathrm{Vol}_u.
\end{aligned}
\]
\end{remark}

\section{Linearization of Differential Operators}\label{sec:linear}

Let $\mathcal E$ be a manifold, let $M \subseteq \mathcal E$ be a submanifold and let $\Delta \in DO (\mathbb R_{\mathcal E})$ be a scalar DO. Denote by $E \to M$ the normal bundle to $M$, i.e.~$E = T\mathcal E|_M /TM$. In this section we show that, under appropriate \emph{linearizability conditions}, the DO $\Delta$ can be \emph{linearized} around $M$ yielding a FWL differential operator $\Delta_{\mathrm{lin}} \in DO_{\mathrm{lin}}(E)$. The DO $\Delta_{\mathrm{lin}}$ represents the first order approximation of $\Delta$ around $M$ in the direction transverse to $M$. This \emph{linearization construction} is a further motivation supporting our definition of FWL DOs.


So let $M \subseteq \mathcal E$ be a submanifold and let $E \to M$ be its normal bundle. We will often consider \emph{adapted coordinates} on $\mathcal E$ around points of $M$, i.e.~coordinates $(X^i, U^\alpha)$ such that $M : \left\{ U^\alpha = 0 \right.$. In particular, the restrictions $(x^i = X^i|_M)$ are coordinates on $M$. From $U^\alpha|_M = 0$ we see that $u^\alpha := dU^\alpha |_M$ are conormal $1$-forms and $(x^i, u^\alpha)$ are vector bundle coordinates on $E$. 

We want to explain what does it mean to \emph{linearize} an order $q$ DO operator $\Delta \in DO_q (\mathbb R_{\mathcal E})$ around $M$. We proceed as follows: 1) first, we recall the linearization of a function, 2) second, we discuss the linearization of a symmetric multivector, and, finally 3) we define the linearization of a generic DO. So, let $F \in C^\infty (\mathcal E)$. We say that $F$ is \emph{linearizable} (around $M$) if $F|_M = 0$. In this case, $dF|_M$ is a conormal $1$-form to $M$, i.e.~a section of the conormal bundle $E^\ast \hookrightarrow T^\ast \mathcal E|_M$. Hence it corresponds to a FWL function on $E$. We put $F_{\mathrm{lin}} = \ell_{dF|_M}$ and call it the \emph{linearization} of $F$. For instance, if $(X^i, U^\alpha)$ are adapted coordinates on $\mathcal E$, then the $(U^\alpha)$ are linearizable and the linear fiber coordinates $(u^\alpha)$ on $E$ are their linearizations. If $F$ is any linearizable function on $\mathcal E$, then locally, around a point of $M$, $F(X,U) = F_\alpha(X) U^\alpha + \mathcal O(U^2)$, for some functions $F_\alpha (X)$ of the $(X^i)$ (given by $F_\alpha (X) = \frac{\partial F}{\partial U^\alpha} (X,0)$), and, in this case, $F_{\mathrm{lin}} = F_\alpha(x)u^\alpha$. Notice that every linear function $\varphi \in C^\infty_{\mathrm{lin}}(E)$ is the linearization of a (non-unique) linearizable function $F \in C^\infty (\mathcal E)$: $\varphi = F_{\mathrm{lin}}$.

We now pass to symmetric multivectors. So, let $P$ be a symmetric $q$-multivector on $\mathcal E$. We say that $P$ is \emph{linearizable} if it belongs to the ideal $\mathcal I_M$ in $\mathfrak X^\bullet_{\mathrm{sym}}(\mathcal E)$ spanned by vector fields that are tangent to $M$. In other words $M$ is a \emph{coisotropic submanifold} of $\mathcal E$ with respect to $P$. This notion of linearizable multivector agrees with that of \emph{FWLizable tensor} described in \cite[Definition 5.4]{PSV2020}.

\begin{proposition}\label{prop:lin_multiv}
Let $P \in \mathfrak X^q_{\mathrm{sym}}(\mathcal E)$ be a linearizable symmetric multivector on $\mathcal E$. Then, there exists a unique FWL symmetric $q$-multivector $P_{\mathrm{lin}}$ on $E$ such that
\begin{equation}\label{eq:lin_multiv}
P_{\mathrm{lin}} ((F_1)_{\mathrm{lin}}, \ldots, (F_q)_{\mathrm{lin}}) = P(F_1, \ldots, F_q)_{\mathrm{lin}}
\end{equation}
for all linearizable functions $F_i \in C^\infty (\mathcal E)$. The \emph{linearization} $P \mapsto P_{\mathrm{lin}}$ preserves the Poisson bracket of symmetric multivectors.
\end{proposition}

\begin{proof}
We begin remarking that, as $P \in \mathcal I_M$, the function $P(F_1, \ldots, F_q)$ is clearly linearizable for any choice of linearizable functions $F_i$. Now we want to show that the rhs of (\ref{eq:lin_multiv}) does only depend on $(F_i)_{\mathrm{lin}}$. We do this in coordinates. So, let $(X^i, U^\alpha)$ be adapted coordinates on $\mathcal E$, and let $(x^i, u^\alpha)$ be the associated vector bundle coordinates on $E$. Locally
\begin{equation}\label{eq:P}
P = \sum_{l+m = q}P^{i_1\cdots i_l,\alpha_1 \cdots \alpha_m}(X,U) \frac{\partial}{\partial X^{i_1}} \odot \cdots \odot \frac{\partial}{\partial X^{i_l}} \odot \frac{\partial}{\partial U^{\alpha_1}} \odot \cdots \odot \frac{\partial}{\partial U^{\alpha_m}}.  
\end{equation}
Hence
\[
P(F_1, \ldots, F_q) =  \sum_{l+m = q} \sum_{\sigma \in S_{l,m}}P^{i_1\cdots i_l,\alpha_1 \cdots \alpha_m}(X,U) \frac{\partial F_{\sigma(1)}}{\partial X^{i_1}} \cdots \frac{\partial F_{\sigma(l)}}{\partial X^{i_l}} \frac{\partial F_{\sigma(l+1)}}{\partial U^{\alpha_1}} \cdots \frac{\partial F_{\sigma(l+m)}}{\partial U^{\alpha_m}}. 
\]
It follows from $P \in \mathcal I_M$ that $P^{\alpha_1, \ldots, \alpha_q}(X,0) = 0$. Now compute
\[
P(F_1, \ldots, F_q)_{\mathrm{lin}} = \left(\frac{\partial}{\partial U^\alpha}|_{(x,0)}P(F_1, \ldots, F_q)\right)u^\alpha.
\]
Using that $F_i|_M = 0$, and that $P^{\alpha_1, \ldots, \alpha_q}(X,0) = 0$, we find
\[
\begin{aligned}
\frac{\partial}{\partial U^\alpha}|_{(x,0)}P(F_1, \ldots, F_q) & =  \frac{\partial P^{\alpha_1 \cdots \alpha_q}}{\partial U^\alpha}(x,0) (F_1)_{\alpha_1} (x) \cdots (F_q)_{\alpha_q} (x) \\
& \quad + \sum_{\sigma \in S_{1, k-1}} P^{i,\alpha_1 \cdots \alpha_{q-1}}(x,0) \frac{\partial (F_{\sigma(1)})_\alpha}{\partial x^{i}} (x,0) (F_{\sigma(2)})_{\alpha_1} \cdots (F_{\sigma(q)})_{\alpha_{q-1}}
\end{aligned}
\]
which does only depend on the $(F_i)_{\mathrm{lin}}$. As every FWL function is a linearization, this also shows that $P_{\mathrm{lin}}$ is well-defined on linear functions. Finally, $P_{\mathrm{lin}}$ can be uniquely extended to all functions on $E$, as a symmetric $q$-multivector, also denoted by $P_{\mathrm{lin}}$, and locally given by 
\[
P_{\mathrm{lin}} = P^{\alpha_1 \cdots \alpha_q}_{\mathrm{lin}}(x, u) \frac{\partial}{\partial u^{\alpha_1}} \odot \cdots \odot \frac{\partial}{\partial u^{\alpha_q}} 
+ P^{i,\alpha_1 \cdots \alpha_{q-1}}(x,0) \frac{\partial}{\partial x^{i}} \odot \frac{\partial}{\partial u^{\alpha_1}} \odot \cdots \odot \frac{\partial}{\partial u^{\alpha_{q-1}}}.
\]
In particular $P_{\mathrm{lin}}$ is a linear multivector.

For the last part of the statement, first notice that the ideal $\mathcal I_M$ is preserved by the Poisson bracket, so, if $P,Q \in \mathcal I_M$, then it makes sense to linearize $\{ P, Q\}$. The rest follows easily from Equation (\ref{eq:lin_multiv}).
\end{proof}

\begin{remark}
Proposition \ref{prop:lin_multiv} is a ``symmetric multivector analogue'' of the following well-known fact. Let $(\mathcal P, \pi)$ be a Poisson manifold, let $(T^\ast \mathcal P)_\pi$ be its cotangent algebroid and let $M \subseteq \mathcal P$ be a coisotropic submanifold. Then the conormal bundle $N^\ast M \subseteq T^\ast \mathcal P$ is a subalgebroid $(T^\ast \mathcal P)_\pi$, hence the normal bundle $NM$ is equipped with a linear Poisson structure $\pi_{\mathrm{lin}}$ (see, e.g., \cite[Section 3]{W1988}).
\end{remark}

By definition, the multivector $P_{\mathrm{lin}}$ in Proposition \ref{prop:lin_multiv} is the \emph{linearization} of $P$. 

We finally come to a generic differential operator $\Delta \in DO_q (\mathbb R_{\mathcal E})$.

\begin{definition}
An order $q$ DO $\Delta \in DO (\mathbb R_{\mathcal E})$ is \emph{order $q$ linearizable} (around $M$) if its symbol $\sigma (\Delta) \in \mathfrak X^q_{\mathrm{sym}}(\mathcal E)$ is linearizable.
\end{definition}

\begin{remark} \quad
\begin{itemize}
\item An order $q-1$ DO $\Delta \in DO (\mathbb R_{\mathcal E})$ is always order $q$ linearizable, but it might not be order $q-1$ linearizable.
\item A function $F \in C^\infty (\mathcal E) = DO_0 (\mathbb R_{\mathcal E})$ is order $0$ linearizable if and only if $F|_M = 0$, i.e.~$F$ is a linearizable function.
\item A vector field $X \in \mathfrak X (\mathcal E) \subset DO_1 (\mathbb R_{\mathcal E})$ is order $1$ linearizable if and only if it is tangent to $M$, i.e.~it is a linearizable vector field.
\end{itemize}
\end{remark}

Now, let $\Delta \in DO_q (\mathbb R_{\mathcal E})$ be an order $q$ linearizable DO. In order to \emph{linearize} it, we use Theorem \ref{theor:iso_DO_D_sym}. In other words, from $\Delta$ we cook up a $q$-$L_E$-multivector $(P_{\mathrm{lin}}, \Phi)$ of $L_{E^\ast} = E^\ast \times_M L$, with $L = \wedge^\mathrm{top} E$ as in Section \ref{sec:FWL_DO}. The construction is inspired by the proof of Theorem \ref{theor:iso_DO_D_sym} itself. We let $P_{\mathrm{lin}} \in \mathfrak X^q_{\mathrm{sym}, \mathrm{lin}}(E)$ be the linearization of the symbol $P = \sigma (\Delta)$ of $\Delta$. It remains to define the vector bundle map
\[
\Phi : S^{q-1} E^\ast \to DL.
\]
For $\varphi_1, \ldots, \varphi_{q-1} \in \Gamma (E^\ast)$, we define $\Phi (\varphi_1, \ldots, \varphi_{q-1})$ as follows. Pick linearizable functions $F_i \in C^\infty (\mathcal E)$ such that $(F_i)_{\mathrm{lin}} = \ell_{\varphi_i}$ and consider the function on $M$
\[
f_{F_1, \ldots, F_{q-1}} := [[\cdots[\Delta, F_1], \cdots], F_{q-1}](1)|_M.
\]
We put
\[
\Phi (\varphi_1, \ldots, \varphi_{k-1}) (U)= {P_{\mathrm{lin}}(\varphi_1, \ldots, \varphi_{q-1}, -)} U + f_{F_1, \ldots, F_{k-1}} U,
\]
for all $U \in \Gamma (L)$. In the rhs, we interpret $P_{\mathrm{lin}}$ as a multiderivation of $E^\ast$. We have to show that $\Phi$ is well defined, i.e.~$f_{F_1, \ldots, F_{q-1}}$ does only depend on the $\varphi_i$, and $\Phi$ is indeed a vector bundle map $\Phi : S^{q-1} E^\ast \to DL$ as desired, i.e.~it is (symmetric and) $C^\infty (M)$-linear in its arguments. To see that $f_{F_1, \ldots, F_{q-1}}$ does not depend on the choice of the $F_i$, assume that $(F_{q-1})_{\mathrm{lin}} = 0$. Hence we have $F_{q-1}|_M = 0$, and $dF_{q-1}|_M = 0$. Additionally,
\[
\begin{aligned}
f_{F_1, \ldots, F_{q-1}} & = [[\cdots[\Delta, F_1], \cdots], F_{q-1}](1)|_M 
\\
& = \left(F_{q-1}[[\cdots[\Delta, F_1], \cdots], F_{q-2}](1) - [[\cdots[\Delta, F_1], \cdots], F_{q-2}](F_{q-1})\right)|_M \\
& = - [[\cdots[\Delta, F_1], \cdots], F_{q-2}](F_{q-1})|_M
\end{aligned}
\] 
Now, $[[\cdots[\Delta, F_1], \cdots], F_{q-2}]$ is a second order DO. As the first jet of $F_{q-1}$ vanishes on $M$, the expression $ [[\cdots[\Delta, F_1], \cdots], F_{q-2}](F_{q-1})|_M$ does only depend on the symbol of $[[\cdots[\Delta, F_1], \cdots], F_{q-2}]$ (and the Hessian of $F_{q-1}$). But
\[
\sigma \left( [[\cdots[\Delta, F_1], \cdots], F_{q-2}] \right) = \sigma (\Delta) (F_1, \ldots, F_{q-2}, -,-),
\]
so, if $P = \sigma (\Delta)$ is locally given by (\ref{eq:P}), using that $P^{\alpha_1 \cdots \alpha_q}(X,0) = 0$, we find that, locally,
\[
f_{F_1, \ldots, F_{q-1}}  \propto P^{i,\alpha_1 \cdots \alpha_{q-2}\alpha_{q-1}}(x,0)  (F_1)_{\alpha_1}(x) \cdots (F_{q-2})_{\alpha_{q-2}}(x) \frac{\partial }{\partial x^{i} } (F_{q-1})_{\alpha_{q-1}}(x) = 0. 
\]
We conclude that $\Phi (\varphi_1, \ldots, \varphi_{q-1})$ is well-defined. It remains to see that $\Phi$ is $C^\infty (M)$-linear in one, hence in all, of its arguments. So, take $f \in C^\infty (M)$. The simple formula:
\begin{equation}\label{eq:FG_lin}
(FG)_{\mathrm{lin}} = F|_M G_{\mathrm{lin}}, \quad \text{for all $F \in C^\infty (\mathcal E)$ and all linearizable $G \in C^\infty (\mathcal E)$},
\end{equation}
shows that, when we replace $\varphi_{q-1}$ with $f \varphi_{q-1}$ in $\Phi(\varphi_1, \ldots, \varphi_{q-1})$, we can replace $F_{q-1}$ with $F F_{q-1}$, where $F\in C^\infty (\mathcal F)$ is any function such that $F|_M = f$. So we have
\[
 \Phi (\varphi_1, \ldots, \varphi_{q-2}, f \varphi_{q-1}) U = {P_{\mathrm{lin}}(\varphi_1, \ldots, \varphi_{q-2}, f \varphi_{q-1}, -)}U + f_{F_1, \ldots, F_{q-2}, F F_{q-1}}U. 
\]
For the first summand, recall from the proof of Theorem \ref{theor:iso_DO_D_sym} (Equations (\ref{eq:3}), (\ref{eq:1})) that
\[
{P_{\mathrm{lin}}(\varphi_1, \ldots, \varphi_{q-2}, f \varphi_{q-1}, -)}U = f {P_{\mathrm{lin}}(\varphi_1, \ldots, \varphi_{q-2}, \varphi_{q-1}, -)}U - l_{P_{\mathrm{lin}}} (\varphi_1, \ldots, \varphi_{q-1})(f) U
\]
For the second summand, compute
\[
\begin{aligned}
f_{F_1, \ldots, F_{q-2}, F F_{q-1}}  & = [[\cdots[\Delta, F_1], \cdots, F_{q-2}], F F_{q-1}](1)|_M \\
& = F [[[\cdots[\Delta, F_1], \cdots], F_{q-2}], F_{q-1}](1)|_M + [[[\cdots[\Delta, F_1], \cdots], F_{q-2}], f](F_{q-1})|_M \\
& = f f_{F_1, \ldots, F_{q-1}} + [[[[\cdots[\Delta, F_1], \cdots], F_{q-2}], F],F_{q-1}](1)|_M \\
& =  F f_{F_1, \ldots, F_{q-1}} + \sigma (\Delta) (F_1, \ldots, F_{q-1}, F)|_M.
\end{aligned}
\]
Now, to conclude the proof of the $C^\infty (M)$-multilinearity, it is enough to show that
\begin{equation}\label{eq:sigma|_M_l}
\sigma (\Delta) (F_1, \ldots, F_{q-1}, F)|_M = l_{P_{\mathrm{lin}}} (\varphi_1, \ldots, \varphi_{q-1})(f).
\end{equation}
But this follows easily from the fact that $P_{\mathrm{lin}}$ is the linearization of $\sigma (\Delta)$ and formula (\ref{eq:FG_lin}) again.

Finally, by construction, $l \circ \Phi = l_{P_{\mathrm{lin}}}$, so the pair $(P_{\mathrm{lin}}, \Phi)$ corresponds to a $q$-$L_E$-multivector. We call the linear DO $A^{-1}(P_{\mathrm{lin}}, \Phi)$ the \emph{linearization} of $\Delta$ and denote it $ \Delta_{\mathrm{lin}}$.

The above discussion leads to the following
\begin{theorem}\label{theor:linear}
The assignment $\Delta \mapsto \Delta_{\mathrm{lin}}$ is a well-defined (\emph{linearization}) map from order $q$ linearizable DOs on $\mathcal E$ and order $q$ FWL DOs on $E$. The linearization preserves the commutator of DOs in the following sense: let $\Delta \in DO_{q+1} (\mathcal E)$ and $\Delta' \in DO_{q'+1}(\mathcal E)$ be an order $q + 1$ linearizable and an order $q' +1$ linearizable DO, respectively. Then 1) $[\Delta, \Delta']$ is order $q+q'+1$ linearizable and 2) its linearization is $[\Delta_{\mathrm{lin}}, \Delta{}'_{\mathrm{lin}}]$.
\end{theorem}

\begin{proof}
We only need to prove the last part of the statement. To do that, recall that the commutator $[\Delta, \Delta']$ is a DO of order $q + q' + 1$ and its symbol is the Poisson bracket $\{ \sigma(\Delta), \sigma (\Delta') \}$. Let $\Delta_{\mathrm{lin}} = A^{-1}(P_{\mathrm{lin}}, \Phi)$, and $\Delta'{}_{\mathrm{lin}} = A^{-1}(P'{}_{\mathrm{lin}}, \Phi')$ be their linearizations. As the ideal $\mathcal I_M$ in $\mathfrak X^\bullet (\mathcal E)$ is preserved by the Poisson bracket, then $[\Delta, \Delta']$ is order $q + q' + 1$ linearizable as desired. Let $[\Delta, \Delta']_{\mathrm{lin}} = A^{-1} (P''{}_{\mathrm{lin}}, \Phi'')$ be its linearization. As the linearization of multivectors preserves the Poisson bracket, then $P''{}_{\mathrm{lin}} = \left\{P_{\mathrm{lin}}, P'{}_{\mathrm{lin}}\right\}$. Finally, we have to take care of $\Phi''$. It is a straightforward computation that we sketch to point out the possible subtleties related to the properties of the various objects involved. Recall from the discussion preceding the statement of the theorem that
\[
\Phi (\varphi_1, \ldots, \varphi_{q}) U = {P_{\mathrm{lin}} (\varphi_1, \ldots, \varphi_{q})}U + f_{F_1, \ldots, F_{q}}U
\]
where 
\[
f_{F_1, \ldots, F_{q}} = [\cdots [\Delta, F_1], \cdots, F_q](1)|_M,
\]
and likewise for $\Phi', \Phi''$. Here $F_i$ is any linearizable function on $\mathcal E$ such that $(F_i)_{\mathrm{lin}} = \varphi_i$. In the following, to stress that, actually, the function $ f_{F_1, \ldots, F_{q}}$ does only depend on the $\varphi_i$, we write $f_{\varphi_1, \ldots, \varphi_q}$ (instead of $f_{F_1, \ldots, F_{q}}$). Additionally, we call $f$ (resp.~$f', f''$) the \emph{$f$-component} of the FWL DO $\Delta_{\mathrm{lin}}$ (resp.~$\Delta'{}_{\mathrm{lin}}, [\Delta_{\mathrm{lin}}, \Delta'{}_{\mathrm{lin}}]$). We have to show that the $f$-component $f^{[\Delta, \Delta']}$ of $[\Delta, \Delta']_{\mathrm{lin}}$ agrees with $f''$. To do this, we compute $f^{[\Delta, \Delta']}$ explicitly. By symmetry and $\mathbb R$-multilinearity, it is enough to evaluate it on $q + q'$ equal sections $\varphi_1 = \cdots = \varphi_{q+q'} = \varphi \in \Gamma (E^\ast)$. So let $F \in C^\infty (\mathcal E)$ be a linearizable function such that $\varphi = F_{\mathrm{lin}}$. We have
\[
\begin{aligned}
f^{[\Delta, \Delta']}_{\underset{\text{$q+q'$ times}}{\underbrace{\varphi, \ldots, \varphi}}} & = [\cdots [[ \Delta, \Delta'], \underset{\text{$q + q'$ times}}{\underbrace{F], \cdots, F]}}(1)|_M \\
& = \sum_{i + j = q + q'} \frac{1}{i!j!}\bigg[[\cdots [ \Delta, \underset{\text{$i$ times}}{\underbrace{F], \cdots, F]}}, [\cdots [ \Delta', \underset{\text{$j$ times}}{\underbrace{F], \cdots, F]}} \bigg](1)|_M.
\end{aligned}
\]
The only terms that survive are those with $i = q-1, q, q+1 $ (hence $j = q' +1, q', q'-1$ respectively). We call them $T_{q-1}, T_q, T_{q+1}$ respectively. The first one is given by
\[
\begin{aligned}
T_{q-1} & =  \frac{1}{(q-1)!(q'+1)!}\bigg[[\cdots [ \Delta, \underset{\text{$q-1$ times}}{\underbrace{F], \cdots, F]}}, [\cdots [ \Delta', \underset{\text{$q' + 1$ times}}{\underbrace{F], \cdots, F]}} \bigg](1)|_M \\
& = \frac{1}{(q-1)!(q'+1)!} f_{\underset{\text{$q - 1$ times}}{\underbrace{\varphi, \ldots, \varphi}}, P'{}_{\mathrm{lin}}(\underset{\text{$q' + 1$ times}}{\underbrace{\varphi, \ldots, \varphi}})}
\end{aligned}.
\]
Similarly the third one is
\[
T_{q+1} = -\frac{1}{(q+1)!(q'-1)!} f'_{\underset{\text{$q' - 1$ times}}{\underbrace{\varphi, \ldots, \varphi}}, P_{\mathrm{lin}}(\underset{\text{$q + 1$ times}}{\underbrace{\varphi, \ldots, \varphi}})}.
\]
To compute $T_q$, we use a simple trick: for every two scalar DOs $\square, \square'$ we have
\[
[\square, \square'](1) = [\square, \square'(1)](1) - [\square', \square(1)](1).
\]
After using this formula we get
\[
\begin{aligned}
T_{q} & =  \frac{1}{q!q'!}\bigg[[\cdots [ \Delta, \underset{\text{$q$ times}}{\underbrace{F], \cdots, F]}}, [\cdots [ \Delta', \underset{\text{$q'$ times}}{\underbrace{F], \cdots, F]}} \bigg](1)|_M \\
& =  \frac{1}{q!q'!}\sigma(\Delta)\bigg(\underset{\text{$q$ times}}{\underbrace{F, \cdots, F}}, [\cdots [ \Delta', \underset{\text{$q'$ times}}{\underbrace{F], \cdots, F]}} (1)\bigg)|_M \\
& \quad - \frac{1}{q!q'!}\sigma(\Delta')\bigg(\underset{\text{$q'$ times}}{\underbrace{F, \cdots, F}}, [\cdots [ \Delta, \underset{\text{$q$ times}}{\underbrace{F], \cdots, F]}} (1)\bigg)|_M .
\end{aligned}
\]
Now we use that both $\Delta$ and $F$ are linearizable to replace $[\cdots [ \Delta', F], \cdots, F](1)$ with its restriction to $M$ in the first summand (likewise for $\Delta$ in the second summand). We get
\[
\begin{aligned}
T_{q} 
& =  \frac{1}{q!q'!}\sigma(\Delta)\bigg(\underset{\text{$q$ times}}{\underbrace{F, \cdots, F}}, f'_{\underset{\text{$q'$ times}}{\underbrace{\varphi, \ldots, \varphi}}}\bigg)|_M - \frac{1}{q!q'!}\sigma(\Delta')\bigg(\underset{\text{$q'$ times}}{\underbrace{F, \cdots, F}}, f_{\underset{\text{$q$ times}}{\underbrace{\varphi, \ldots, \varphi}}}\bigg)|_M \\
& =  \frac{1}{q!q'!}l_{P_{\mathrm{lin}}}\bigg(\underset{\text{$q$ times}}{\underbrace{\varphi, \cdots, \varphi}}\bigg)\big(f'_{\underset{\text{$q'$ times}}{\underbrace{\varphi, \ldots, \varphi}}}\big) - \frac{1}{q!q'!}l_{P'{}_{\mathrm{lin}}}\bigg(\underset{\text{$q'$ times}}{\underbrace{\varphi, \cdots, \varphi}}\bigg)\big(f_{\underset{\text{$q$ times}}{\underbrace{\varphi, \ldots, \varphi}}}\big) \end{aligned}
\]
where we also used (\ref{eq:sigma|_M_l}). At this point, it can be checked directly, explointing the explicit formula (\ref{eq:Poisson_V-multivectors}) for the Poisson bracket of $L_E$-multivectors, that $T_{q-1} + T_q + T_{q+1}$ agrees with the $f$-component $f''$ of $[\Delta_{\mathrm{lin}}, \Delta'{}_{\mathrm{lin}}]$. We conclude that $[\Delta_{\mathrm{lin}}, \Delta'{}_{\mathrm{lin}}] = [\Delta, \Delta']_{\mathrm{lin}}$ as desired.
 \end{proof}
 
\begin{remark}
Let $\mathcal E$ be the total space of a vector bundle $E \to M$ and interpret $M$ as a submanifold in $\mathcal E = E$ via the zero section. In this situation, the normal bundle $NM$ to $M$ identifies canonically with $E$ itself. Clearly, an order $q$ FWL DO $\Delta \in DO_{q, \mathrm{lin}}(E)$ is automatically order $q$ linearizable and it easily follows from the proof of Theorem \ref{theor:linear} that the vector bundle isomorphism $E \cong NM$ identifies $\Delta$ with its own linearization $\Delta_{\mathrm{lin}}$.
 \end{remark}

 \textbf{Acknowledgements.} LV is member of the GNSAGA of INdAM. 

\end{document}